\documentclass[12pt]{article}
\usepackage{amssymb,amsmath,amsthm,tikz,multirow,array,arydshln,pdflscape}
\usetikzlibrary{arrows}

\title{Edge-to-edge Tilings of the Sphere by Angle Congruent Pentagons}
\author{Hoiping Luk, Min Yan\thanks{Research was supported by Hong Kong RGC General Research Fund 16303515.} \\ 
Hong Kong University of Science and Technology}

\usepackage[hidelinks, hyperindex]{hyperref}

\hyperref[sec:function]{}

\DeclareMathOperator{\opt}{\, |\,}

\newcommand{\ssum}{\textstyle \sum}

\newtheorem{theorem}{Theorem}
\newtheorem{lemma}[theorem]{Lemma}

\newtheorem{proposition}[theorem]{Proposition}

\theoremstyle{definition}

\newtheorem*{definition*}{Definition}
\newtheorem*{case*}{Case}
\newtheorem*{example*}{Example}

\numberwithin{equation}{section}

\begin{document}

\maketitle

\begin{abstract}
We develop a systematic method for computing the angle combinations at all vertices in an edge-to-edge tiling of the sphere by pentagons with the same five angles. The method is a useful and necessary step in many tiling problems about pentagonal tilings of the sphere. As an application, we find all edge-to-edge tilings of the sphere by angle congruent pentagons, that allow free and continuous choice of two angle values. 
\end{abstract}

\medskip


\section{Introduction}

Two polygons are {\em angle congruent} if there is a one-to-one correspondence between the edges, such that adjacent edges correspond to adjacent edges, and the angles between corresponding adjacent edges are equal. Similarly, the two polygons are {\em edge congruent} if, instead of the angles, the lengths of the corresponding edges are equal. Two polygons are {\em congruent} if they are angle and edge congruent. 

We recently completed the classification of edge-to-edge tilings of the sphere by congruent pentagons \cite{ay,awy,cly5,gsy,ua,wy1,wy2,wy3}, or by congruent quadrilaterals \cite{cly4}. For triangular tilings, the classification was started by Sommerville \cite{so} a century ago, and was completed by Ueno and Agaoka \cite{ua} in 2002. Therefore the whole classification of edge-to-edge tilings of the sphere by congruent polygons is complete.

The classification work assumes the tiles are angle and edge congruent. The goal of this paper is to see how far we can go by assuming only the angle congruence. In fact, all the results in this paper do not even require the edges to be straight (i.e., great arcs on the sphere). Moreover, most results in this paper only assume all pentagonal tiles have the same five angles, and do not even require the angles to be arranged in the same way (as in angle congruence). The only information we use is the sphere, the pentagon, and that the angles at any vertex should add up to $2\pi$. 

We develop a systematic method for obtaining the numerical information about the angle combinations in the pentagon and at all vertices, in edge-to-edge tilings of the sphere by angle congruent pentagons. The method was already used in our earlier classification work. As an application, we classify all angle congruent tilings, such that two angles can be freely and continuously chosen. 

In Sections 5 and 6 of \cite{gsy}, Gao, Shi, Yan described the angle combinations for the dodecahedron tiling, which has minimal number of $12$ pentagonal tiles. Moreover, we know that any non-minimal pentagonal tiling has $f\ge 16$ tiles \cite{yan}. Therefore we will always assume $f\ge 16$ in this paper. Moreover, since we only consider tilings that are given by naturally embedded graphs, we will only consider edge-to-edge tilings with all vertices having degree $\ge 3$.

Our results are presented as what we call AVCs ({\em anglewise vertex combinations}, first introduced in \cite{gsy}). A typical example is 
\begin{align*}
\{
84\alpha^2\beta\gamma\delta &\colon 
(72-y_1)\alpha\beta\gamma,\,
32\alpha^3,\,
(12+y_1)\beta\delta^3 \opt
y_1\alpha\gamma^3,\,
(12-y_1)\gamma^2\delta^3
\}, \\
&
\alpha=\tfrac{2}{3}\pi,\,
\beta=\tfrac{6}{7}\pi,\,
\gamma=\tfrac{10}{21}\pi,\,
\delta=\tfrac{8}{21}\pi.
\end{align*}
Here we have $84$ pentagonal tiles. The notation $\alpha^2\beta\gamma\delta$ means the angles in the pentagon are $\alpha,\alpha,\beta,\gamma,\delta$, with the given values. There are five possible angle combinations $\alpha\beta\gamma$, $\alpha^3$, $\beta\delta^3$, $\alpha\gamma^3$, $\gamma^2\delta^3$ at the vertices. The coefficients $72-y_1$, $32$, $12+y_1$, $y_1$, $12-y_1$ are the numbers of the respective vertices. The divider $\opt$ separates the five combinations into two groups. The three combinations before the divider are {\em necessary} in the sense that they must appear in the tiling, i.e., the coefficients are positive. The two combinations after the divider are {\em optional} in the sense that these are the only vertices that may appear in addition to the necessary three, i.e., the coefficients are non-negative. We may choose $y_1$ to be any integer, such that all the coefficients in the necessary part are positive and all the coefficients in the optional part are non-negative. The condition means exactly $0\le y_1\le 12$, and we have total of $13$ choices.

Our derivation of an AVC starts from its degree $3$ part. This part is provided by Table \ref{deg3AVC}, and is denoted $\text{AVC}_3$. For the AVC above, we have $\text{AVC}_3=\{\alpha^2\beta\gamma\delta \colon \alpha\beta\gamma,\,\alpha^3\}$. By \cite[Lemma 8]{wy1}, there is at most one angle not appearing at degree $3$ vertices, and the angle must appear at a vertex of degree $4$ or $5$ in specific way. In the specific example, the angle $\delta$ appears at $\beta\delta^3$. Then we construct the {\em full} AVC by assuming $\alpha\beta\gamma,\alpha^3,\beta\delta^3$ are vertices. We indicate the starting point of the construction by writing $\text{AVC}\supset\{\alpha\beta\gamma,\alpha^3,\beta\delta^3\}$.

In Section \ref{high}, we first present some basic results about angles in pentagonal tilings. In particular, for up to three distinct angles, Proposition \ref{3complete} provides the initial $\text{AVC}_3$s consisting of the angle combinations in the pentagon and the possible degree $3$ vertices, and serve as the starting point for the calculation of full AVCs.

In Sections \ref{routine1}, \ref{routine2} and \ref{routine3}, we describe the method for deriving the full AVCs from the initial ones provided in Section \ref{high}. The method in Sections \ref{routine1} and \ref{routine2} are rather routine. In Section \ref{routine3}, we give four examples not covered by the two routines. 

In Section \ref{classify}, we list all the full AVCs with up to three distinct angles at degree $3$ vertices. Moreover, we outline how to get all the full AVCs with four distinct angles at degree $3$ vertices. We also explore the AVCs with five distinct angles at degree $3$ vertices, where the methods in Sections \ref{routine1}, \ref{routine2}, \ref{routine3} may not be sufficient.

In Section \ref{realize}, we study the realizations of the AVCs by actual angle congruent. We concentrate only on those AVCs allowing free and continuous choice of two angles. The list in Section \ref{classify} shows that, beyond the minimal $f=12$, there are only three such AVCs, with respective $f=24,36,60$. Then in Theorem \ref{anglerealize}, we show that only the AVCs with $f=24,60$ can be realized. In fact, the realizations are the {\em pentagonal subdivisions} of Platonic solids, in \cite{wy1,wy3,yan2}. We further discuss what the result means in terms of tilings by (angle and edge) congruent pentagons, including the possibility that the edges are not straight.

\section{Angles in Pentagonal Tiling of the Sphere}
\label{high}

In this paper, a {\em tiling} is an edge-to-edge tiling of the sphere by pentagons with the same five angles, such that all vertices have degree $\ge 3$, and the number of tiles $f\ge 16$. The reason for $f\ge 16$ will be explained below.

Let $v,e,f$ be the numbers of vertices, edges, and tiles. Let $v_k$ be the number of vertices of degree $k$. We have
\begin{align*}
2
&=v-e+f, \\
2e=5f
&=\ssum_{k=3}^{\infty}kv_k=3v_3+4v_4+5v_5+\cdots, \\
v
&=\ssum_{k=3}^{\infty}v_k=v_3+v_4+v_5+\cdots.
\end{align*}
Then it is easy to derive $2v=3f+4$ and  
\begin{align}
\tfrac{f}{2}-6
&=\ssum_{k\ge 4}(k-3)v_k=v_4+2v_5+3v_6+\cdots, \label{vcountf} \\
v_3
&=20+\ssum_{k\ge 4}(3k-10)v_k=20+2v_4+5v_5+8v_6+\cdots. \label{vcountv}
\end{align}

The equality \eqref{vcountv} shows that most vertices have degree $3$. We call vertices of degree $>3$ {\em high degree} vertices.

By \eqref{vcountf}, $f$ is an even integer $\ge 12$. Edge-to-edge tilings of the sphere by $12$ congruent pentagons have been classified \cite{ay, gsy}. Moreover, Gao, Shi and Yan \cite{gsy} also classified most tilings by $12$ {\em angle congruent} pentagons. The only exception is tilings by  pentagon with angles $\alpha,\alpha,\alpha,\beta,\gamma$, for which the number of angle congruent tilings is too many to be listed. By \cite{yan}, we also know $f\ne 14$. Therefore we will assume that $f$ is even and $f\ge 16$ throughout this paper.

The sum of all angles at a vertex (called {\em angle sum} of the vertex) is $2\pi$. The following is the {\em angle sum for pentagon}, from \cite[Lemma 4]{wy1}.

\begin{lemma}\label{anglesum}
For a tiling by $f$ pentagons with the same five angles, the sum of five angles in the pentagon is $(3 + \frac{4}{f})\pi$.
\end{lemma}

Since most vertices have degree $3$, the appearance (and disappearance) of an angle at all degree $3$ vertices has implication on how the angle appears in the pentagon. The following two lemmas are \cite[Lemmas 6, 7, 8]{wy1}.

\begin{lemma}\label{more3}
In a tiling by pentagons with the same five angles,
\begin{itemize}
\item If an angle appears at every degree $3$ vertex, then the angle appears at least twice in the pentagon. 
\item If an angle appears at least twice at every degree 3 vertex, then the angle appears at least three times in the pentagon.
\end{itemize}
\end{lemma}

The second statement also applies to two angles together appearing twice at every degree $3$ vertex. Then the two angles together appear at least three times in the pentagon.

\begin{lemma}\label{hdeg}
Suppose in a tiling by pentagons with the same five angles, an angle $\theta$ does not appear at degree $3$ vertices.
\begin{enumerate}
\item There can be at most one such angle $\theta$.
\item The angle $\theta$ appears only once in the pentagon.
\item One of $\alpha\theta^3$, $\theta^4$, $\theta^5$ is a vertex, where $\alpha\ne\theta$.
\end{enumerate}
\end{lemma}

The following is the {\em counting lemma} \cite[Lemma 5]{cly5}. It actually applies to more general tilings (not just by pentagons, see \cite[Lemma 4]{cly4}). The idea behind the result is central to Rao's classification of convex pentagons that can tile the plane \cite{rao}. 

\begin{lemma}\label{counting}
Suppose $\theta,\rho$ appear the same number of times in the pentagon. If all vertices are of the form $\theta^k\rho^l\cdots$, with $k\le l$ and no $\theta,\rho$ in the remainder, then all vertices are of the form $\theta^k\rho^k\cdots$, with no $\theta,\rho$ in the remainder. 
\end{lemma}

We remark that Lemmas \ref{anglesum}, \ref{more3}, \ref{hdeg}, \ref{counting} do not assume all tiles have the same angle arrangement, i.e., angle congruent.

The following is \cite[Lemma 1]{awy}, and is the starting point of the calculations in this paper.

\begin{lemma}\label{deg3v}
If an edge-to-edge tiling of a surface has at most five distinct angles at degree $3$ vertices, then after suitable relabelling of the distinct angles, the anglewise vertex combinations at degree $3$ vertices is in Table \ref{deg3AVC}.
\end{lemma}

\begin{table}[h]
\centering
\begin{tabular}{|c|c|c|c|}
\hline 
\multicolumn{3}{|c|}{Necessary} & Optional  \\
\hline 
\hline  
$\alpha^3$ & \multicolumn{1}{r}{} & & \\
\hline 
\hline  
$\alpha\beta^2$ & \multicolumn{1}{r}{} & & \\
\hline 
\hline  
\multicolumn{2}{|c|}{$\alpha\beta\gamma$} & & $\alpha^3$ \\
\cline{1-2} \cline{4-4} 
\multirow{2}{*}{$\alpha\beta^2$} & $\alpha^2\gamma$ & & \\
\cline{2-2}  
 & $\gamma^3$ & & \\
\hline 
\hline  
\multicolumn{2}{|c|}{\multirow{5}{*}{$\alpha\beta\gamma$}}& \multirow{2}{*}{$\alpha\delta^2$}  &  $\beta^2\delta$ \\
\cline{4-4}
\multicolumn{2}{|c|}{} &&  $\beta^3$  \\
\cline{3-4} 
\multicolumn{2}{|c|}{} & \multirow{2}{*}{$\alpha^2\delta$}  & $\beta\delta^2$ \\
\cline{4-4}
\multicolumn{2}{|c|}{} &&  $\beta^3$ \\
\cline{3-4} 
\multicolumn{2}{|c|}{} & $\delta^3$ &  \\
\hline 
\multirow{2}{*}{$\alpha\beta^2$} & \multicolumn{2}{|c|}{$\gamma\delta^2$}  & $\alpha^2\delta$ \\
\cline{2-4} 
 & \multicolumn{2}{|c|}{$\alpha^2\gamma,\delta^3$} &  \\
\hline
\hline 
\multirow{5}{*}{$\alpha\beta\gamma$} & 
\multicolumn{2}{|c|}{\multirow{5}{*}{$\alpha\delta\epsilon$}}
& $\beta\delta^2,\beta^2\epsilon$ \\
\cline{4-4}
& \multicolumn{2}{|c|}{} &  $\beta\delta^2,\gamma\epsilon^2,\alpha^3$ \\
\cline{4-4}
& \multicolumn{2}{|c|}{} &  $\beta\delta^2,\gamma^2\epsilon$ \\
\cline{4-4}
& \multicolumn{2}{|c|}{} &  $\beta\delta^2,\gamma^3$ \\
\cline{4-4}
& \multicolumn{2}{|c|}{} &  $\beta\delta^2,\epsilon^3$ \\
\hline
\end{tabular}
\quad
\begin{tabular}{|c|c|c|c|}
\hline 
\multicolumn{3}{|c|}{Necessary} & Optional  \\
\hline 
\hline 
\multirow{19}{*}{$\alpha\beta\gamma$} 
& \multirow{12}{*}{$\alpha\delta^2$} 
& \multirow{3}{*}{$\alpha^2\epsilon$} 
& $\beta\epsilon^2$ \\
\cline{4-4}
&&& $\beta^2\delta$ \\
\cline{4-4}
&&& $\beta^3$ \\ 
\cline{3-4} 
&& \multirow{3}{*}{$\beta\epsilon^2$} 
& $\alpha^2\epsilon$ \\
\cline{4-4}
&&& $\gamma^2\delta$ \\
\cline{4-4}
&&& $\gamma^3$ \\
\cline{3-4} 
&& \multirow{4}{*}{$\beta^2\epsilon$}  
& $\gamma\epsilon^2$ \\
\cline{4-4}
&&& $\gamma^2\delta$ \\
\cline{4-4}
&&& $\gamma^3$ \\
\cline{4-4}
&&& $\delta\epsilon^2$ \\
\cline{3-4} 
&& \multirow{2}{*}{$\delta\epsilon^2$} 
& $\beta^2\epsilon$ \\
\cline{4-4}
&&& $\beta^3$ \\
\cline{3-4} 
&& $\epsilon^3$ & $\beta^2\delta$ \\
\cline{2-4} 
& \multirow{6}{*}{$\alpha^2\delta$}
& \multirow{3}{*}{$\beta^2\epsilon$} 
& $\alpha\epsilon^2$ \\
\cline{4-4}
&&& $\gamma\delta^2$ \\
\cline{4-4}
&&& $\gamma^3$ \\
\cline{3-4} 
&& \multirow{2}{*}{$\delta^2\epsilon$} 
& $\beta\epsilon^2$ \\
\cline{4-4}
&&& $\beta^3$ \\
\cline{3-4} 
&& $\epsilon^3$ & $\beta\delta^2$ \\
\cline{2-4} 
& \multicolumn{2}{|c|}{$\delta\epsilon^2$} &   $\alpha^3$ \\
\hline 
\multicolumn{2}{|c|}{\multirow{3}{*}{$\alpha\beta^2,\gamma\delta^2$}}
& \multirow{2}{*}{$\alpha^2\epsilon$} &  $\beta\gamma^2$ \\
\cline{4-4} 
\multicolumn{2}{|c|}{} &&  $\delta\epsilon^2$ \\
\cline{3-4} 
\multicolumn{2}{|c|}{} & $\epsilon^3$ &  $\alpha^2\delta$ \\
\hline 
\end{tabular}
\caption{$\text{AVC}_3$ for up to five distinct angles.}
\label{deg3AVC}
\end{table}

We remark that Lemma \ref{deg3v} is about general tilings of any surface, and does not assume pentagon or sphere. We may further use Lemmas \ref{more3} and \ref{hdeg} to find possible combinations of angles in the pentagon as well as at degree $3$ vertices. The following gives the complete answer for up to three distinct angles at degree $3$ vertices.

\begin{proposition}\label{3complete}
For tilings of the sphere by more than $12$ pentagons with the same five angles, with up to three distinct angles at degree $3$ vertices, the angle combination in the pentagon and at all degree $3$ vertices are, up to relabelling of angles, given by Table \ref{3completeAVC}.
\end{proposition}

\begin{table}[h]
\centering
\begin{tabular}{|c|c|}
\hline 
$\text{AVC}_3$ & angle combination of pentagon \\
\hline \hline 
$\alpha^3$ &
$\alpha^4\beta$ \\
\hline 
$\alpha\beta^2$ &
$\alpha^2\beta^3$ \\
\hline 
\multirow{2}{*}{$\alpha\beta\gamma,\alpha^3$} &
$\alpha^2\beta^2\gamma$ \\
\cline{2-2}
& $\alpha^2\beta\gamma\delta$  \\
\hline 
\multirow{2}{*}{$\alpha\beta^2,\alpha^2\gamma$} &
$\alpha^3\beta\gamma,\alpha^2\beta^2\gamma,\alpha^2\beta\gamma^2$ \\
\cline{2-2}
& $\alpha^2\beta\gamma\delta$ \\
\hline 
\multirow{2}{*}{$\alpha\beta^2,\gamma^3$} &
$\alpha^2\beta^2\gamma,\alpha^2\beta\gamma^2,\alpha\beta^3\gamma,\alpha\beta\gamma^3$ \\
\cline{2-2}
& $\alpha\beta^2\gamma\delta,\alpha\beta\gamma^2\delta$ \\
\hline
\end{tabular}
\caption{Angles in the pentagon and at degree $3$ vertices, for up to three distinct angles at degree $3$ vertices.}
\label{3completeAVC}
\end{table}

The left column of the table is all the possible $\text{AVC}_3$s, i.e., sets of angle combinations at degree $3$ vertices. The right column is split into the case of all angles appearing at degree $3$ vertices, and the case one angle not appearing at degree $3$ vertices. For example, if $\alpha\beta\gamma,\alpha^3$ are all the degree $3$ vertices, then the AVC contains one of the following
\begin{align*}
&\{\alpha^2\beta^2\gamma\colon \alpha\beta\gamma,\alpha^3\}, &
&\{\alpha^2\beta\gamma\delta \colon \alpha\beta\gamma,\alpha^3,\alpha\delta^3\}, &
&\{\alpha^2\beta\gamma\delta \colon \alpha\beta\gamma,\alpha^3,\beta\delta^3\}, \\
&\{\alpha^2\beta\gamma\delta \colon \alpha\beta\gamma,\alpha^3,\delta^4\}, &
&\{\alpha^2\beta\gamma\delta \colon \alpha\beta\gamma,\alpha^3,\delta^5\}. &&
\end{align*}
Lemma \ref{hdeg} gives the necessary vertices $\alpha\delta^3,\beta\delta^3,\delta^4,\delta^5$ in four of five cases. In fact, the lemma also gives $\gamma\delta^3$. The case becomes the third one after the symmetry of exchanging $\beta$ and $\gamma$.

\begin{proof}
The proof follows the five $\text{AVC}_3$ in Table \ref{deg3AVC} for up to three distinct angles. We need to consider two cases:
\begin{enumerate}
\item All angles appear at degree $3$ vertices.
\item One angle does not appear at degree $3$ vertices.
\end{enumerate}

Consider $\{\alpha^3\}$ from the first row of Table \ref{deg3AVC}. In the first case, the pentagon must be $\alpha^5$. The angle sum of $\alpha^3$ and the angle sum for pentagon (Lemma \ref{anglesum}) imply
\[
3\alpha=2\pi,\quad
5\alpha=(3+\tfrac{4}{f})\pi.
\]
Then we get $f=12$, which can be dismissed because we only consider $f\ge 16$. 

In the second case, let $\beta$ be the extra angle appearing only at high degree vertices. By Lemma \ref{hdeg}, we have $\{\alpha^3,\alpha\beta^3\},\{\alpha^3,\beta^4\},\{\alpha^3,\beta^5\}$. The angle sums of the two vertices imply $\alpha=\frac{2}{3}\pi$ and $\beta=\frac{4}{9}\pi,\frac{1}{2}\pi,\frac{2}{5}\pi$. Then we consider the angle combinations $\alpha^k\beta^{5-k}$ for the pentagon, such that the angle sum for pentagon  
\[
k\alpha+(5-k)\beta=(3+\tfrac{4}{f})\pi
\]
yields even $f\ge 16$. The only possibility is $\alpha^4\beta$, which corresponds to $f=36,24,60$ for the three choices of $\beta$.

Consider $\{\alpha\beta^2\}$ from the second row of Table \ref{deg3AVC}. Since $\alpha$ appears at every degree $3$ vertex, and $\beta$ appears twice at every degree $3$ vertex, by Lemma \ref{more3}, we know $\alpha$ appears at least twice in the pentagon, and $\beta$ appears at least three times in the pentagon. Therefore the pentagon has the angle combination $\alpha^2\beta^3$.

Next we assume three angles $\alpha,\beta,\gamma$ appear at degree $3$ vertices. In the second case, we denote by $\delta$ the extra angle not appearing at degree $3$ vertices.

Consider $\{\alpha\beta\gamma \opt \alpha^3\}$ from Table \ref{deg3AVC}. If $\alpha\beta\gamma$ is the only degree $3$ vertex, then by Lemma \ref{more3}, we know each of $\alpha,\beta,\gamma$ appears twice in the pentagon, a contradiction. Therefore $\alpha^3$ is necessarily a vertex, and $\text{AVC}_3=\{\alpha\beta\gamma, \alpha^3\}$. Then $\alpha$ appears at every degree $3$ vertex. By Lemma \ref{more3}, we know $\alpha$ appears at least twice in the pentagon. In the first case, up to the symmetry of exchanging $\beta$ and $\gamma$, this means that the pentagon is $\alpha^3\beta\gamma$ or $\alpha^2\beta^2\gamma$. In the second case, the pentagon is $\alpha^2\beta\gamma\delta$. However, the angle sums of $\alpha\beta\gamma, \alpha^3$ and the angle sum for the pentagon $\alpha^3\beta\gamma$ imply 
\[
(3+\tfrac{4}{f})\pi
=3\alpha+\beta+\gamma
=\tfrac{2}{3}\cdot 3\alpha+(\alpha+\beta+\gamma)
=(\tfrac{2}{3}+2)\pi.
\]
Then we get $f=12$, and the case is dismissed. 

Consider $\{\alpha\beta^2,\alpha^2\gamma\}$ from Table \ref{deg3AVC}. Since $\alpha$ appears at every degree $3$ vertex, by Lemma \ref{more3}, we know $\alpha$ appears at least twice in the pentagon. In the first case, this means that the pentagon is $\alpha^3\beta\gamma,\alpha^2\beta^2\gamma,\alpha^2\beta\gamma^2$. In the second case, the pentagon is $\alpha^2\beta\gamma\delta$.

Consider $\{\alpha\beta^2,\gamma^3\}$ from Table \ref{deg3AVC}. By Lemma \ref{more3}, we know $\beta,\gamma$ together must appear at least three times in the pentagon. In the first case, this means the pentagon is $\alpha^2\beta^2\gamma,\alpha^2\beta\gamma^2,\alpha\beta^3\gamma,\alpha\beta^2\gamma^2,\alpha\beta\gamma^3$. In the second case, the pentagon is $\alpha\beta^2\gamma\delta,\alpha\beta\gamma^2\delta$. However, the angle sums of $\alpha\beta^2,\gamma^3$ and the angle sum for pentagon $\alpha\beta^2\gamma^2$ imply
\[
(3+\tfrac{4}{f})\pi
=\alpha+2\beta+2\gamma
=(\alpha+2\beta)+\tfrac{2}{3}\cdot 3\gamma
=(2+\tfrac{2}{3})\pi.
\]
Then we get $f=12$, and the case is dismissed. 
\end{proof}

\section{AVCs with Finitely Many $f$}
\label{routine1}

Starting from each combination in Proposition \ref{3complete}, we find the angle combinations at all the high degree vertices. We may also calculate the numbers of tiles and vertices.

Let us consider $\text{AVC}\supset\{\alpha^2\beta\gamma\delta\colon \alpha\beta\gamma,\alpha^3,\beta\delta^3\}$, which means $\alpha$, $\alpha$, $\beta$, $\gamma$, $\delta$ are the angles of the pentagon, $\alpha\beta\gamma,\alpha^3$ are all the degree $3$ vertices, and $\beta\delta^3$ is also a vertex.

The angle sums of $\alpha\beta\gamma,\alpha^3,\beta\delta^3$ and the angle sum for the pentagon $\alpha^2\beta\gamma\delta$ are
\[
\alpha+\beta+\gamma
=3\alpha
=\beta+3\delta
=2\pi,\quad
2\alpha+\beta+\gamma+\delta=(3+\tfrac{4}{f})\pi.
\]
Then we get
\[
\alpha=\tfrac{2}{3}\pi,\,
\beta=(1-\tfrac{12}{f})\pi,\,
\gamma=(\tfrac{1}{3}+\tfrac{12}{f})\pi,\,
\delta=(\tfrac{1}{3}+\tfrac{4}{f})\pi.
\]
For the angles to be positive (already implied by $f\ge 16$) and distinct, we require $f\ne 24,36$.

Besides the existing vertices $\alpha\beta\gamma,\alpha^3,\beta\delta^3$, any other vertex $\alpha^a\beta^b\gamma^c\delta^d$ is given by a quadruple $(a,b,c,d)$ of non-negative integers satisfying the angle sum equation 
\[
\tfrac{2}{3}a
+(1-\tfrac{12}{f})b
+(\tfrac{1}{3}+\tfrac{12}{f})c
+(\tfrac{1}{3}+\tfrac{4}{f})d=2,
\]
and the high degree condition
\[
a+b+c+d\ge 4.
\]
The angle sum equation can be rephrased as an expression for $f$
\[
f=\frac{12(-3b+3c+d)}{2a+3b+c+d-6}.
\]
By $f\ge 16$, the angle sum equation also implies
\[
2\ge \tfrac{2}{3}a
+\tfrac{1}{4}b
+\tfrac{1}{3}c
+\tfrac{1}{3}d.
\]
We substitute the finitely many quadruples of non-negative integers satisfying the inequality above and the high degree condition into the formula for $f$. We keep only those quadruples yielding even integers $f\ge16$ that are not $24$ and $36$. They are listed in Table \ref{2abcd_abc|3a|b3d_vertex}.

Note that we should also consider the possibility that the substitution yields $f=\frac{0}{0}$, which means $3b=3c+d$ and $2a+3b+c+d=6$. Since this contradicts the high degree condition, we do not need to consider the possibility.

\begin{table}[h]
\centering
\scalebox{1}{
\begin{tabular}{|c |c |c |c |c |c |c |}
\hline
\multirow{2}{*}{number} & 
\multicolumn{4}{|c|}{vertex $\alpha^a\beta^b\gamma^c\delta^d$} & 
\multirow{2}{*}{$f$} \\
\cline{2-5}
& $a$ & $b$ & $c$ & $d$ &   \\
\hline \hline
$x_1$ & 1 & 1 & 1 & 0 &  \\
\cline{1-5} 
$x_2$ & 3 & 0 & 0 & 0 &  \\
\cline{1-5} 
$x_3$ & 0 & 1 & 0 & 3 &  \\
\hline \hline
$y_1$ & 0 & 0 & 2 & 2 & 48 \\
\hline 
$y_1$ & 1 & 0 & 1 & 2 & \multirow{3}{*}{$60$}  \\
\cline{1-5} 
$y_2$ & 0 & 0 & 3 & 1 &   \\
\cline{1-5} 
$y_3$ & 0 & 0 & 0 & 5 &   \\
\hline
$y_1$ & 0 & 0 & 4 & 0 & 72 \\
\hline 
$y_1$ & 1 & 0 & 3 & 0 & \multirow{2}{*}{$84$}  \\
\cline{1-5} 
$y_2$ & 0 & 0 & 2 & 3 &   \\
\hline 
$y_1$ & 1 & 0 & 2 & 1 & \multirow{2}{*}{$108$}  \\
\cline{1-5} 
$y_2$ & 0 & 0 & 1 & 4 &   \\
\hline
$y_1$ & 0 & 0 & 3 & 2 & 132 \\
\hline
$y_1$ & 0 & 0 & 4 & 1 & 156 \\
\hline
$y_1$ & 0 & 0 & 0 & 5 & 180 \\
\hline 
\end{tabular}
}
\caption{Vertices for $\{\alpha^2\beta\gamma\delta\colon \alpha\beta\gamma,\alpha^3,\beta\delta^3\}$.}
\label{2abcd_abc|3a|b3d_vertex} 
\end{table}

Next we find the number of each vertex. In the table, $x_1,x_2,x_3$ are the numbers of existing vertices $\alpha\beta\gamma,\alpha^3,\beta\delta^3$. For $f=48$, $y_1$ is the number of $\gamma^2\delta^2$, and $\alpha\beta\gamma,\alpha^3,\beta\delta^3,\gamma^2\delta^2$ are all the vertices. By the angle combination $\alpha^2\beta\gamma\delta$ of the pentagon, the total numbers of $\alpha,\beta,\gamma,\delta$ in the tiling are respectively $2\times 48,48,48,48$. On the other hand, we may count the total numbers from the viewpoint of the vertices. This gives the {\em angle counting equations}
\begin{align*}
2\times 48 &=x_1+3x_2, \\
48 &=x_1+x_3, \\
48 &=x_1+2y_1, \\
48 &=3x_3+2y_1.
\end{align*}
The unique solution $x_1=36$, $x_2=20$, $x_3=12$, $y_1=6$ gives the full AVC. We also get the specific angle values by substituting $f=48$:
\begin{align*}
\{
48\alpha^2\beta\gamma\delta &\colon 
36\alpha\beta\gamma,\,
20\alpha^3,\,
12\beta\delta^3 \opt
6\gamma^2\delta^2
\}, \\
&
\alpha=\tfrac{2}{3}\pi,\,
\beta=\tfrac{3}{4}\pi,\,
\gamma=\tfrac{7}{12}\pi,\,
\delta=\tfrac{5}{12}\pi.
\end{align*}
For $f=60$, the angle counting equations are
\begin{align*}
2\times 60 &=x_1+3x_2+y_1, \\
60 &=x_1+x_3, \\
60 &=x_1+y_1+3y_2, \\
60 &=3x_3+2y_1+y_2+5y_3.
\end{align*}
The solution is
\[
x_1=60-y_1-3y_2,\,
x_2=20+y_2,\,
x_3=y_1+3y_2,\,
y_1+2y_2+y_3=12.
\]
The last equation is obtained by eliminating $x_1,x_2,x_3$ from the four equations, and is satisfied by finitely many (total number $=1+3+\cdots+13=49$) triples $(y_1,y_2,y_3)$ of non-negative integers. Among these triples, the only one not yielding all positive $x_1,x_2,x_3$ is $(y_1,y_2,y_3)=(0,0,12)$. With the implicit understanding on the 48 possible choices of $(y_1,y_2,y_3)$, we denote the full AVC as
\begin{align*}
\{
60\alpha^2\beta\gamma\delta &\colon 
(60-y_1-3y_2)\alpha\beta\gamma,\,
(20+y_2)\alpha^3,\,
(y_1+3y_2)\beta\delta^3 \\
\opt &
y_1\alpha\gamma\delta^2,\,
y_2\gamma^3\delta,\,
(12-y_1-2y_2)\delta^5
\}, \\
&
\alpha=\tfrac{2}{3}\pi,\,
\beta=\tfrac{4}{5}\pi,\,
\gamma=\tfrac{18}{15}\pi,\,
\delta=\tfrac{2}{5}\pi.
\end{align*}
The full AVCs for all the $f$ in Table \ref{2abcd_abc|3a|b3d_vertex} are given in Section \ref{classify3}.

Strictly speaking, we still need to consider the possibility that, for $f$ not listed in the table, the existing vertices may already form a viable full AVC. However, applying Lemma \ref{counting} to $\beta,\gamma$ in $\{\alpha^2\beta\gamma\delta\colon \alpha\beta\gamma,\alpha^3,\beta\delta^3\}$, we get a contradiction.

Now we describe the routine process in general. For an angle combination $\alpha^a\beta^b\gamma^c\cdots$, we call the collection of orders $(a,b,c,\dots)^T$ its {\em order vector} (always written vertically). In the example above, the pentagon $\alpha^2\beta\gamma\delta$ has the order vector $P=(2,1,1,1)^T$, and the existing vertices $\alpha\beta\gamma,\alpha^3,\beta\delta^3$ have the order vectors $X_1=(1,1,1,0)^T$, $X_2=(3,0,0,0)^T$, $X_3=(0,1,0,3)^T$. The given data form the {\em order matrix}
\[
(P\, X)=\begin{pmatrix}
2 & 1 & 3 & 0 \\
1 & 1 & 0 & 1 \\
1 & 1 & 0 & 0 \\
1 & 0 & 0 & 3 
\end{pmatrix}.
\]

Let (the horizontal vector) $\vec{\alpha}=(\alpha,\beta,\gamma,\dots)$ be the distinct angles in the pentagon. The angle sum for pentagon and the angle sum of vertices are
\[
\vec{\alpha}P=(3+\tfrac{4}{f})\pi,\quad
\vec{\alpha}X_i=2\pi.
\] 
This is the same as
\[
\vec{\alpha}(P\, X)
=\pi(3,2,2,2,\dots)
+\tfrac{1}{f}\pi(4,0,0,0,\dots).
\]

Suppose the order matrix $(P\, X)$ is invertible. Then we get the angle values
\[
\vec{\alpha}=\vec{\alpha}_0+\tfrac{1}{f}\vec{\alpha}_1,
\]
where
\[
\vec{\alpha}_0=\pi(3,2,2,2,\dots)(P\,X)^{-1},\quad
\vec{\alpha}_1=\pi(4,0,0,0,\dots)(P\,X)^{-1},
\]
are the constant part and $\frac{1}{f}$ part of the angle values. The second equality means
\begin{equation}\label{eq1}
\vec{\alpha}_1P=4\pi,\quad
\vec{\alpha}_1X=(0,0,0,\dots).
\end{equation}

Next we verify the requirement that $\alpha,\beta,\gamma,\dots$ have positive and distinct values, which often means that $f$ cannot take certain finitely many values. Moreover, we may estimate the first angle $\alpha=\alpha_0+\frac{1}{f}\alpha_1$ as follows:
\begin{enumerate}
\item If $\alpha_0>0$, then $\alpha>0$ for sufficiently large $f$. 
\item If $\alpha_0<0$, then $\alpha>0$ for only finitely many $f$. 
\item If $\alpha_0=0$, then $\alpha=\frac{1}{f}\alpha_1$ can become arbitrarily small. 
\end{enumerate}
In the first and second cases, we get a positive lower bound for $\alpha$. The third is the only case $\alpha$ has no positive lower bound. The same estimation can be carried out for the other angles. 

Next we try to find all the order vectors $\vec{v}=(a,b,c,\dots)^T$ other than $X_i$. The vector satisfies the angle sum equation
\[
2\pi=a\alpha+b\beta+c\gamma+\dots
=\vec{\alpha}\vec{v}
=\vec{\alpha}_0\vec{v}+\tfrac{1}{f}\vec{\alpha}_1\vec{v},
\]
and the high degree condition
\[
\deg\vec{v}=a+b+c+\dots\ge 4.
\]

Suppose all coordinates of $\vec{\alpha}_0$ are nonzero. Then we have positive lower bounds for all the angles. By the angle sum equation, these lower bounds translate into upper bounds for the coordinates of $\vec{v}$, which means finitely many possible choices for $\vec{v}$. We substitute those choices satisfying the high degree condition $\deg\vec{v}\ge 4$ into 
\[
f=\frac{\vec{\alpha}_1\vec{v}}{2\pi-\vec{\alpha}_0\vec{v}},
\]
and keep only those yielding even integers $f\ge 16$, and such that all angles are positive and distinct. Then we get finitely many possible $f$ and for each $f$, finitely many optional vertices.

We need to consider the possibility that $\vec{\alpha}_0\vec{v}=2\pi$ and $\vec{\alpha}_1\vec{v}=0$. Since the order matrix is invertible, we may write $\vec{v}=(P\,X)\vec{u}$. Then we get
\begin{align*}
2\pi
&=\vec{\alpha}_0\vec{v}
=\vec{\alpha}_0(P\,X)\vec{u}
=\pi(3,2,2,2,\dots)\vec{u}, \\
0
&=\vec{\alpha}_1\vec{v}
=\vec{\alpha}_1(P\,X)\vec{u}
=\pi(4,0,0,0,\dots)\vec{u}.
\end{align*}
Therefore $\vec{u}=(0,\lambda_1,\lambda_2,\dots)$, with $\lambda_1+\lambda_2+\cdots=1$. If all the existing vertices $X_i$ have degree $3$, then $\vec{v}=(P\,X)\vec{u}=\lambda_1 X_1+\lambda_2 X_2+\cdots$ also has degree $3$, and fails the high degree condition. This means we may dismiss the possibility that $\vec{\alpha}_0\vec{v}=2\pi$ and $\vec{\alpha}_1\vec{v}=0$.

In the earlier example that started with $\{\alpha^2\beta\gamma\delta\colon \alpha\beta\gamma,\alpha^3,\beta\delta^3\}$, we know $X_1,X_2$ have degree $3$, and $X_3$ has degree $4$. Then the argument above is not valid. Still, for all the cases we calculate, where one degree $4$ or $5$ vertex is added via Lemma \ref{hdeg}, the possibility that $\vec{\alpha}_0\vec{v}=2\pi$ and $\vec{\alpha}_1\vec{v}=0$ does not happen.

For each possible $f$, the order vectors of the corresponding optional vertices form a matrix $Y$. The numbers $x_i$ and $y_j$ of all vertices satisfy the angle counting equation
\[
fP=X\vec{x}+Y\vec{y},\quad
\vec{x}=(x_1,x_2,\dots)^T,\quad
\vec{y}=(y_1,y_2,\dots)^T.
\]
The equation is the same as 
\[
(P\, X)\begin{pmatrix}
f \\ -\vec{x}
\end{pmatrix}
=fP-X\vec{x}
=Y\vec{y}.
\]
Then we get
\[
\begin{pmatrix}
f \\ -\vec{x}
\end{pmatrix}
=(P\, X)^{-1}Y\vec{y}.
\]
Therefore $\vec{y}$ satisfies one equation (recall $f$ has specific value), and $\vec{x}$ can be expressed in terms of $\vec{y}$. 

Finally, if all the existing vertices have degree $3$, then we claim that the equation for $\vec{y}$ is exactly the vertex counting equation. Indeed, multiplying $\vec{\alpha}$ and $(1,1,\dots)$ respectively to the angle counting equation, we get
\begin{align*}
(3+\tfrac{4}{f})f\pi
&=f\vec{\alpha}P
=\vec{\alpha}X\vec{x}+\vec{\alpha}Y\vec{y} \\
&=2\pi(1,1,\dots)\vec{x}+2\pi(1,1,\dots)\vec{y}, \\
5f
&=f(1,1,\dots)P
=(1,1,\dots)X\vec{x}+(1,1,\dots)Y\vec{y} \\
&=3(1,1,\dots)\vec{x}+(1,1,\dots)Y\vec{y}.
\end{align*}
Eliminating $\vec{x}$ from the two equalities, we get exactly \eqref{vcountf}
\[
\tfrac{f}{2}-6=[(1,1,\dots)Y-3(1,1,\dots)]\vec{y}.
\]
Therefore the vertex counting equation can be derived from the angle counting equations.

If there is an angle not appearing at degree $3$ vertices, then the last column of $X$ is given by the third part of Lemma \ref{hdeg}, and the vertex counting equation also includes the number of this vertex. For example, in the earlier example, for $f=60$, the vertex counting equation is 
\[
x_3+y_1+y_2+3y_3=\tfrac{60}{2}-6=24. 
\]
On the other hand, the last angle counting equation gives us 
\[
x_3=20-\tfrac{1}{3}(2y_1+y_2+5y_3).
\]
Substituting this into the vertex counting equation above, we get
\[
y_1+2y_2+y_3=12.
\]
This is a general way of getting the relation among $y_i$.

\section{AVCs with Variable $f$}
\label{routine2}

The routine in Section \ref{routine1} assumes that the order matrix $(P\,X)$ is invertible, and the constant part $\vec{\alpha}_0$ of every angle is nonzero. In this section, we still assume the ivertibility of the order matrix, but allow {\em one} angle to have zero constant part. 

Let us consider the $\text{AVC}\supset\{\alpha\beta\gamma^2\delta\colon \alpha\beta^2,\gamma^3,\alpha\delta^3\}$, which means $\alpha$, $\beta$, $\gamma$, $\gamma$, $\delta$ are the angles of the pentagon, $\alpha\beta^2,\gamma^3$ are all the degree $3$ vertices, and $\alpha\delta^3$ is also a vertex.

The order matrix 
\[
(P\, X)=\begin{pmatrix}
1 & 1 & 0 & 1 \\
2 & 2 & 0 & 0 \\
1 & 0 & 3 & 0 \\
1 & 0 & 0 & 3
\end{pmatrix}
\]
is invertible. The angle sums of $\alpha\beta^2,\gamma^3,\alpha\delta^3$ and the  angle sum for the pentagon $\alpha\beta\gamma^2$ imply
\[
\alpha=\tfrac{24}{f}\pi,\,
\beta=(1-\tfrac{12}{f})\pi,\,
\gamma=\tfrac{2}{3}\pi,\,
\delta=(\tfrac{2}{3}-\tfrac{8}{f})\pi.
\]
We note that $\vec{\alpha}_0=(0,1,\frac{2}{3},\frac{2}{3})$, and $\alpha$ has zero constant part. For the angles to be positive (implied by $f\ge 16$) and distinct, we require $f\ne 36,48$.

The angle sum of a vertex $\alpha^a\beta^b\gamma^c\delta^d$ implies
\[
a=\tfrac{f}{72}(6-3b-2c-2d)+\tfrac{1}{6}(3b+2d),\quad
2\ge \tfrac{1}{4}b+\tfrac{2}{3}c+\tfrac{1}{6}d.
\]
Substituting the finitely many triples $(b,c,d)$ of non-negative integers satisfying the inequality into the expression for $a$, we get all the possible optional vertices in Tables \ref{ab2cd||1a2b|3c|1a3d_vertex2} and \ref{ab2cd||1a2b|3c|1a3d_vertex}. 

We note that, if $6-3b-2c-2d=0$, then we need to consider whether $a=\frac{1}{6}(3b+2d)$ is a non-negative integer. This would give vertices that may appear for all $f$. It turns out that all we get are the existing vertices $\alpha\beta^2,\gamma^3,\alpha\delta^3$. 

Table \ref{ab2cd||1a2b|3c|1a3d_vertex2} consists of those $(b,c,d,f)$ satisfying $6-3b-2c-2d<0$, $f$ is a fixed even integer satisfying $f\ge 16$ and $f\ne 36,48$, $a$ is a non-negative integer, and the degree of the vertex is $\ge 4$. It turns out we get $a=0$ in all cases.

\begin{table}[htp]
\centering
\scalebox{1}{
\begin{tabular}{|c |c |c |c |c |c |}
\hline
number & $a$ & $b$ & $c$ & $d$ & $f$  \\
\hline \hline 
$y_1$ & 0 & 2 & 2 & 1 & \multirow{10}{*}{$f=16$} \\
\cline{1-5}
$y_2$ & 0 & 0 & 2 & 4 & \\
\cline{1-5}
$y_3$ & 0 & 4 & 1 & 2 & \\
\cline{1-5}
$y_4$ & 0 & 8 & 0 & 0 & \\
\cline{1-5}
$y_5$ & 0 & 2 & 1 & 5 & \\
\cline{1-5}
$y_6$ & 0 & 6 & 0 & 3 & \\
\cline{1-5}
$y_7$ & 0 & 0 & 1 & 8 & \\
\cline{1-5}
$y_8$ & 0 & 4 & 0 & 6 & \\
\cline{1-5}
$y_9$ & 0 & 2 & 0 & 9 & \\
\cline{1-5}
$y_{10}$ & 0 & 0 & 0 & 12 & \\
\cline{1-6} 
$y_1$ & 0 & 2 & 2 & 0 & \multirow{9}{*}{$f=18$}  \\
\cline{1-5}
$y_2$ & 0 & 4 & 1 & 0 & \\
\cline{1-5}
$y_3$ & 0 & 0 & 2 & 3 & \\
\cline{1-5}
$y_4$ & 0 & 6 & 0 & 0 & \\
\cline{1-5}
$y_5$ & 0 & 2 & 1 & 3 & \\
\cline{1-5}
$y_6$ & 0 & 4 & 0 & 3 & \\
\cline{1-5}
$y_7$ & 0 & 0 & 1 & 6 & \\
\cline{1-5}
$y_8$ & 0 & 2 & 0 & 6 & \\
\cline{1-5}
$y_9$ & 0 & 0 & 0 & 9 & \\
\hline
\end{tabular}
\quad
\begin{tabular}{|c |c |c |c |c |c |}
\hline
number & $a$ & $b$ & $c$ & $d$ & $f$  \\
\hline \hline 
$y_1$ & 0 & 1 & 2 & 1 & \multirow{6}{*}{$f=20$}  \\
\cline{1-5}
$y_2$ & 0 & 5 & 0 & 0 & \\
\cline{1-5}
$y_3$ & 0 & 2 & 1 & 2 & \\
\cline{1-5}
$y_4$ & 0 & 3 & 0 & 3 & \\
\cline{1-5}
$y_5$ & 0 & 0 & 1 & 5 & \\
\cline{1-5}
$y_6$ & 0 & 1 & 0 & 6 & \\
\cline{1-6} 
$y_1$ & 0 & 4 & 0 & 0 & \multirow{6}{*}{$f=24$}  \\
\cline{1-5}
$y_2$ & 0 & 2 & 1 & 1 & \\
\cline{1-5}
$y_3$ & 0 & 0 & 2 & 2 & \\
\cline{1-5}
$y_4$ & 0 & 2 & 0 & 3 & \\
\cline{1-5}
$y_5$ & 0 & 0 & 1 & 4 & \\
\cline{1-5}
$y_6$ & 0 & 0 & 0 & 6 & \\
\cline{1-6} 
$y_1$ & 0 & 1 & 1 & 2 & $f=28$  \\
\cline{1-6} 
$y_1$ & 0 & 2 & 0 & 2 & \multirow{2}{*}{$f=30$}   \\
\cline{1-5}
$y_2$ & 0 & 0 & 0 & 5 &   \\
\hline
\end{tabular}
}
\caption{Vertices with fixed $f$ for $\{\alpha\beta\gamma^2\delta\colon \alpha\beta^2,\gamma^3,\alpha\delta^3\}$.}
\label{ab2cd||1a2b|3c|1a3d_vertex2}
\end{table}

Table \ref{ab2cd||1a2b|3c|1a3d_vertex} consists of vertices that may appear for multiple $f$, i.e., non-fixed $f$. The list includes the three existing vertices $\alpha\beta^2,\gamma^3,\alpha\delta^3$ and those $(b,c,d)$ satisfying $6-3b-2c-2d>0$. Such vertices allow $f$ to be arbitrarily large, but some modulus condition must be satisfied in order for $a$ to be a non-negative integer. Moreover, the high degree condition becomes lower bounds for $f$.

\begin{table}[htp]
\centering
\scalebox{1}{
\begin{tabular}{|c |c |c |c |c |c |c |}
\hline
\multirow{2}{*}{number} & 
\multicolumn{4}{|c|}{vertex $\alpha^a\beta^b\gamma^c\delta^d$} & 
\multirow{2}{*}{$f$} & 
\multirow{2}{*}{$f$ mod 72} \\
\cline{2-5}
 & $a$ & $b$ & $c$ & $d$ & & \\
\hline \hline
$x_1$ & 1 & 2 & 0 & 0 & & \\
\cline{1-5} 
$x_2$ & 0 & 0 & 3 & 0 & & \\
\cline{1-5}
$x_3$ & 1 & 0 & 0 & 3 & & \\
\hline  \hline 
$z_1$ & $\frac{f+60}{72}$ & 1 & 0 & 1 & 
$f=12(72),\, f\ge 84$ & 
$12$ \\
\hline 
$z_2$ & $\frac{f+36}{72}$ & 1 & 1 & 0 & 
$f=36(72),\, f\ge 108$ & 
$36$ \\
\hline 
$z_3$ & $\frac{f}{36}$ & 0 & 2 & 0 & 
$f=0(36),\, f\ge 72$ & 
$0,36$ \\
\hline 
$z_4$ & $\frac{f+24}{36}$ & 0 & 0 & 2 & 
$f=12(36),\, f\ge 84$ & 
$12,48$ \\
\hline 
$z_5$ & $\frac{f+12}{36}$ & 0 & 1 & 1 & 
$f=24(36),\, f\ge 60$ & 
$24,60$ \\
\hline 
$z_6$ & $\frac{f+12}{24}$ & 1 & 0 & 0 & 
$f=12(24),\, f\ge 60$ & 
$12,36,60$ \\
\hline 
$z_7$ & $\frac{f}{18}$ & 0 & 1 & 0 & 
$f=0(18),\, f\ge 54$ & 
$0,18,36,54$ \\
\hline 
$z_8$ & $\frac{f+6}{18}$ & 0 & 0 & 1 & 
$f=12(18),\, f\ge 66$ & 
$12,30,48,66$ \\
\hline 
$z_9$ & $\frac{f}{12}$ & 0 & 0 & 0 & 
$f=0(12),\, f\ge 60$ & 
$0,12,24,36,48,60$ \\
\hline
\end{tabular}
}
\caption{Vertices with variable $f$ for $\{\alpha\beta\gamma^2\delta\colon \alpha\beta^2,\gamma^3,\alpha\delta^3\}$.}
\label{ab2cd||1a2b|3c|1a3d_vertex}
\end{table}

There is no overlapping between the two tables. For example, although $f=18$ appears in Table \ref{ab2cd||1a2b|3c|1a3d_vertex2}, we require $f\ge 54$ for $f=0(18)$ in Table \ref{ab2cd||1a2b|3c|1a3d_vertex}. 

Next, we calculate the numbers of tiles and vertices by solving the angle counting equation. For all $f$ in Table \ref{ab2cd||1a2b|3c|1a3d_vertex2}, it turns out that we always get contradictions. Therefore only the cases of variable $f$ remain.

For all the vertices in Table \ref{ab2cd||1a2b|3c|1a3d_vertex}, the angle counting equations are
\begin{align*}
f
&=x_1+x_3
+\tfrac{f+60}{72}z_1
+\tfrac{f+36}{72}z_2
+\tfrac{f}{36}z_3
+\tfrac{f+24}{36}z_4 \\
&\quad
+\tfrac{f+12}{36}z_5
+\tfrac{f+12}{24}z_6
+\tfrac{f}{18}z_7
+\tfrac{f+6}{18}z_8
+\tfrac{f}{12}z_9, \\
f
&=2x_1+z_1+z_2+z_6, \\
2f
&=3x_2+z_2+2z_3+z_5+z_7, \\
f
&=3x_3+z_1+2z_4+z_5+z_8.
\end{align*}
The solution is expressions of $x_i$ in terms of $f$ and $z_j$:
\begin{align*}
x_1
&=\tfrac{1}{2}f-\tfrac{1}{2}z_1-\tfrac{1}{2}z_2-\tfrac{1}{2}z_6, \\
x_2
&=\tfrac{1}{3}f-\tfrac{1}{3}z_2-\tfrac{2}{3}z_3-\tfrac{1}{3}z_5-\tfrac{1}{3}z_7, \\
x_3
&=\tfrac{1}{3}f-\tfrac{1}{3}z_1-\tfrac{2}{3}z_4-\tfrac{1}{3}z_5-\tfrac{1}{3}z_8,
\end{align*}
and an equality relating $f$ and $z_j$ (and not involving $x_i$), like what we get in Section \ref{routine1}. In fact, the angle counting equations imply the vertex counting equation \eqref{vcountf}
\begin{align*}
\tfrac{f}{2}-6
&=x_3+(\tfrac{f+60}{72}-1)z_1
+(\tfrac{f+36}{72}-1)z_2
+(\tfrac{f}{36}-1)z_3 \\
&\quad
+(\tfrac{f+24}{36}-1)z_4
+(\tfrac{f+12}{36}-1)z_5
+(\tfrac{f+12}{24}-2)z_6\\
&\quad
+(\tfrac{f}{18}-2)z_7
+(\tfrac{f+6}{18}-2)z_8
+(\tfrac{f}{12}-3)z_9.
\end{align*}
Substituting the formula for $x_3$, we get 
\begin{align*}
\tfrac{f}{2}-\tfrac{f}{3}-6
&=(\tfrac{f+60}{72}-1-\tfrac{1}{3})z_1
+(\tfrac{f+36}{72}-1)z_2
+(\tfrac{f}{36}-1)z_3 \\
&\quad
+(\tfrac{f+24}{36}-1-\tfrac{2}{3})z_4
+(\tfrac{f+12}{36}-1-\tfrac{1}{3})z_5
+(\tfrac{f+12}{24}-2)z_6\\
&\quad
+(\tfrac{f}{18}-2)z_7
+(\tfrac{f+6}{18}-2-\tfrac{1}{3})z_8
+(\tfrac{f}{12}-3)z_9.
\end{align*}
It turns out that $f-36$ is a common factor on both sides. By $f\ne 36$, we may divide the factor and get the relation for $z_j$
\[
z_1+z_2+2z_3+2z_4+2z_5+3z_6+4z_7+4z_8+6z_9=12.
\]
The relation does not involve $f$! We will explain that this is a general phenomenon at the end of the section.

We summarize the full AVC as (see page \pageref{eg12})
\begin{align*}
\{
f\alpha\beta\gamma^2\delta &\colon 
x_1\alpha\beta^2, \,
x_2\gamma^3, \,
x_3\alpha\delta^3 \\
\opt &
z_1\alpha^{\frac{f+60}{72}}\beta\delta,\,
z_2\alpha^{\frac{f+36}{72}}\beta\gamma,\,
z_3\alpha^{\frac{f}{36}}\gamma^2,\,
z_4\alpha^{\frac{f+24}{36}}\delta^2,\, 
\\
&
z_5\alpha^{\frac{f+12}{36}}\gamma\delta,\,
z_6\alpha^{\frac{f+12}{24}}\beta,\,
z_7\alpha^{\frac{f}{18}}\gamma,\,
z_8\alpha^{\frac{f+6}{18}}\delta,\,
z_9\alpha^{\frac{f}{12}}
\}, \\
& x_1 =\tfrac{1}{2}f-\tfrac{1}{2}z_1-\tfrac{1}{2}z_2-\tfrac{1}{2}z_6, \\
& x_2 =\tfrac{1}{3}f-\tfrac{1}{3}z_2-\tfrac{2}{3}z_3-\tfrac{1}{3}z_5-\tfrac{1}{3}z_7, \\
& x_3 =\tfrac{1}{3}f-\tfrac{1}{3}z_1-\tfrac{2}{3}z_4-\tfrac{1}{3}z_5-\tfrac{1}{3}z_8, \\
& z_1+z_2+2z_3+2z_4+2z_5+3z_6+4z_7+4z_8+6z_9=12, \\
&
\alpha=\tfrac{24}{f}\pi,\,
\beta=(1-\tfrac{12}{f})\pi,\,
\gamma=\tfrac{2}{3}\pi,\,
\delta=(\tfrac{2}{3}-\tfrac{8}{f})\pi.
\end{align*}
Since the relation for $z_j$ implies some $z_j\ne 0$, the existing vertices $\alpha\beta^2,\gamma^3,\alpha\delta^3$ cannot be all the vertices.

The modulus condition means some $z_j=0$, i.e., the corresponding vertex does not appear. For example, if $f=0$ mod $72$, then only $z_3,z_7,z_9$ can be non-trivial. Denote $f=72k$. Then the AVC becomes 
\begin{align*}
\{
72k\alpha\beta\gamma^2\delta &\colon 
x_1\alpha\beta^2, \,
x_2\gamma^3, \,
x_3\alpha\delta^3 
\opt 
z_3\alpha^{2k}\gamma^2,\,
z_7\alpha^{4k}\gamma,\,
z_9\alpha^{6k}
\}, \\
& x_1 =36k, \\
& x_2 =24k-\tfrac{2}{3}z_3-\tfrac{1}{3}z_7, \\
& x_3 =24k, \\
& 2z_3+4z_7+6z_9=12, \\
&
\alpha=\tfrac{1}{3k}\pi,\,
\beta=(1-\tfrac{1}{6k})\pi,\,
\gamma=\tfrac{2}{3}\pi,\,
\delta=(\tfrac{2}{3}-\tfrac{1}{9k})\pi.
\end{align*}
For the other modulus classes of $f$, we list the  non-trivial $z_j$ in Table \ref{ab2cd||1a2b|3c|1a3d_mod}.

\begin{table}[htp]
\centering
\scalebox{1}{
\begin{tabular}{|c |c |}
\hline
$f$ mod $72$ &
non-trivial $z_j$ \\
\hline \hline
0 &
$z_3,z_7,z_9$ \\
\hline
12 &
$z_1,z_4,z_6,z_8,z_9$ \\
\hline
18, 54 &
$z_7$ \\
\hline
24 &
$z_5,z_9$ \\
\hline
30, 66 &
$z_8$ \\
\hline
36 &
$z_2,z_3,z_6,z_7,z_9$ \\
\hline
48 &
$z_4,z_8,z_9$ \\
\hline
60 &
$z_5,z_6,z_9$ \\
\hline
\end{tabular}
}
\caption{Non-trivial $z_j$ for various modulus classes of $f$.}
\label{ab2cd||1a2b|3c|1a3d_mod}
\end{table}

For each modulus class of $f$, there are finitely many combinations of the corresponding non-trivial $z_j$ satisfying the $z_j$ only relation. Among these, we choose only those $z_j$ such that $x_1,x_2,x_3$ are positive integers. Moreover, all vertices with nonzero $z_j$ should have degree $\ge 4$.

Now we describe the routine process in general. Similar to in Section \ref{routine1}, we assume the order matrix $(P\, X)$ is invertible. Moreover, we assume the constant part of one angle is zero, and the constant parts of all other angles are nonzero.

Let $\alpha$ be the special angle with zero constant part. Then $\alpha=\tfrac{1}{f}\alpha_1$ and
\[
\vec{\alpha}_0=(0,\beta_0,\gamma_0,\dots)=(0,\vec{\beta}_0),\quad
\vec{\alpha}_1=(\alpha_1,\beta_1,\gamma_1,\dots)=(\alpha_1,\vec{\beta}_1).
\]
The order vector $\vec{v}=(a,b,c,\dots)^T=(a,\vec{w})^T$ of a vertex $\alpha^a\beta^b\gamma^c\cdots$ satisfies the angle sum equation
\[
2\pi
=\vec{\alpha}_0\vec{v}+\tfrac{1}{f}\vec{\alpha}_1\vec{v}
=\vec{\beta}_0\vec{w}+\tfrac{1}{f}\alpha_1a+\tfrac{1}{f}\vec{\beta}_1\vec{w}.
\]
We solve the equation for $a$
\[
a=\tfrac{f}{\alpha_1}(2\pi-\vec{\beta}_0\vec{w})-\tfrac{1}{\alpha_1}\vec{\beta}_1\vec{w}.
\]
Since the constant parts of $\beta,\gamma,\dots$ are nonzero, we get lower bounds for these angles like in Section \ref{routine1}. These lower bounds imply finitely many possible choices for $\vec{w}$. Substituting these choices into the formula for $a$, we get three possibilities:
\begin{enumerate}
\item If $2\pi-\vec{\beta}_0\vec{w}>0$, then we get non-negative integers $a$ for sufficiently large $f$ satisfying a modulus condition. 
\item If $2\pi-\vec{\beta}_0\vec{w}<0$, then we get non-negative integers $a$ for only finitely many $f$. 
\item If $2\pi-\vec{\beta}_0\vec{w}=0$, then we get a vertex appearing for all $f$, whenever $a=-\frac{1}{\alpha_1}\vec{\beta}_1\vec{w}$ is a non-negative integer. 
\end{enumerate}

The case $2\pi-\vec{\beta}_0\vec{w}=0$ is exactly the case $\vec{\alpha}_0\vec{v}=2\pi$ and $\vec{\alpha}_1\vec{v}=0$ in Section \ref{routine1}. If all the existing vertices have degree $3$, then the case does not yield any new vertex. Moreover, if the existing vertices include one degree $4$ or $5$ vertex via Lemma \ref{hdeg}, then our experience also shows the case $2\pi-\vec{\beta}_0\vec{w}=0$ does not yield new vertices.

From the case $2\pi-\vec{\beta}_0\vec{w}<0$, for each of the finitely many possible $f$, we get a finite collection of all optional vertices like Table \ref{ab2cd||1a2b|3c|1a3d_vertex2}. Then our experience shows that the corresponding angle counting equations always lead to contradiction.

Therefore only the case $2\pi-\vec{\beta}_0\vec{w}>0$ leads to full AVCs. The order vectors of the optional vertices form the columns of a matrix  
\[
Z=Z_0+fZ_1,
\]
where $Z_0$ (consisting of $(-\tfrac{1}{\alpha_1}\vec{\beta}_1\vec{w},\vec{w})^T$) and $Z_1$ (consisting of $(\tfrac{1}{\alpha_1}(2\pi-\vec{\beta}_0\vec{w}),\vec{0})^T$) are constant matrices. The comparable matrix $Y$ in Section \ref{routine1} is a constant matrix without the $f$ part. The angle sum equation for $Z$ is
\[
2\pi(1,1,\dots)
=\vec{\alpha}Z
=\vec{\alpha}_0Z_0+\vec{\alpha}_1Z_1+f\vec{\alpha}_0Z_1+\tfrac{1}{f}\vec{\alpha}_1Z_0.
\]
Since this is satisfied by infinitely many $f$, we get
\begin{equation}\label{eq2}
\vec{\alpha}_0Z_0+\vec{\alpha}_1Z_1=2\pi(1,1,\dots),\quad
\vec{\alpha}_0Z_1=\vec{\alpha}_1Z_0=(0,0,\dots).
\end{equation}
Like Section \ref{routine1}, the solution of the angle counting equation
\[
fP=X\vec{x}+Z\vec{z}
\]
is
\[
\begin{pmatrix}
f \\ -\vec{x}
\end{pmatrix}
=(P\, X)^{-1}Z\vec{z}.
\]
Again $\vec{z}$ satisfies one equation (without $\vec{x}$), and $\vec{x}$ can be expressed in terms of $\vec{z}$. 

Our example suggests that $f$ can be canceled in the equation satisfied by $\vec{z}$. This is no accident. Multiplying $\vec{\alpha}_1$ to the angle counting equation and using \eqref{eq1}, \eqref{eq2}, we get
\[
4\pi f
=\vec{\alpha}_1fP
=\vec{\alpha}_1X\vec{x}+\vec{\alpha}_1Z\vec{z}
=(\vec{\alpha}_1Z_0+\vec{\alpha}_1Z_1f)\vec{z}
=\vec{\alpha}_1Z_1\vec{z}f.
\]
Therefore $f$ can be cancelled, and $\vec{z}$ satisfies an equation $4\pi=\vec{\alpha}_1Z_1\vec{z}$ with constant coefficients.

\section{Some Complicated Examples}
\label{routine3}

The routine processes in Sections \ref{routine1} and \ref{routine2} cover all except two AVCs in Proposition \ref{3complete}. In this section, we first compute these two AVCs. Then we compute two more examples not covered by Proposition \ref{3complete}. 

\medskip

\noindent{\bf Example 1}. 
$\{\alpha\beta^2\gamma\delta\colon \alpha\beta^2,\gamma^3\}$, no other degree $3$ vertices.

\medskip

By Lemma \ref{hdeg}, one of $\alpha\delta^3,\beta\delta^3,\gamma\delta^3,\delta^4,\delta^5$ is a vertex. The case $\text{AVC}\supset\{\alpha\beta^2\gamma\delta\colon \alpha\beta^2,\gamma^3,\alpha\delta^3\}$ can be handled by Section \ref{routine1}, and the result is given on page \pageref{eg12}. The case $\text{AVC}\supset\{\alpha\beta^2\gamma\delta\colon \alpha\beta^2,\gamma^3,\beta\delta^3\}$ can be handled by Section \ref{routine2}, and the result is given on page \pageref{eg26}. It remains to consider $\{\alpha\beta^2\gamma\delta\colon \alpha\beta^2,\gamma^3,\gamma\delta^3\}$, $\{\alpha\beta^2\gamma\delta\colon \alpha\beta^2,\gamma^3,\delta^4\}$, $\{\alpha\beta^2\gamma\delta\colon \alpha\beta^2,\gamma^3,\delta^5\}$. What is special about the three cases is that the order matrices are singular.

We deal with the last one $\{\alpha\beta^2\gamma\delta\colon \alpha\beta^2,\gamma^3,\delta^5\}$ first. The angle sums of $\alpha\beta^2,\gamma^3,\delta^5$ and the angle sum for pentagon $\alpha\beta^2\gamma\delta$ imply
\[
\alpha+2\beta=2\pi,\,
\gamma=\tfrac{2}{3}\pi,\,
\delta=\tfrac{2}{5}\pi,\,
f=60.
\]

Next, we apply the idea of Lemma \ref{counting}. Given the pentagon $\alpha\beta^2\gamma\delta$, the total number of $\alpha$ in the tiling is $f$, and the total number of $\beta$ is $2f$, twice of the number of $\alpha$. The equality $2a=b$ holds for all the existing vertices $\alpha\beta^2,\gamma^3,\delta^5$. Therefore to maintain the global doubling relation between the total numbers of $\alpha$ and $\beta$, we either have $2a=b$ for all the vertices, or have vertices with $2a>b$ as well as vertices with $2a<b$.

If $\alpha<\beta$, then $\alpha<\frac{2}{3}\pi<\beta$, and the angle sum of $\alpha^a\beta^b\gamma^c\delta^d$ implies
\[
2>\tfrac{2}{3}b+\tfrac{2}{3}c+\tfrac{2}{5}d.
\]
If $2a<b$, then $b\ge 1$. The following are all the triples $(b,c,d)$ satisfying $b\ge 1$ and the inequality above
\[
(2,0,1), (2,0,0), (1,1,1), (1,1,0), (1,0,3), (1,0,2), (1,0,1), (1,0,0). 
\]
Then the only ways we can add $a$ to satisfy $2a<b$ and $a+b+c+d\ge 4$ are $(a,b,c,d)=(1,2,0,1),(0,1,0,3)$. By $\alpha+2\beta=2\pi$, we cannot have $(1,2,0,1)$. On the other hand, if $(0,1,0,3)$ is a vertex, then the angle sums of $\alpha\beta^2,\delta^5,\beta\delta^3$ imply $\beta=2\delta$ and $\alpha=\delta$, contradicting the distinct angle assumption. We conclude the existing vertices $\alpha\beta^2,\gamma^3,\delta^5$ are all the vertices.

If $\alpha>\beta$, then $\alpha>\frac{2}{3}\pi>\beta$, and the angle sum of $\alpha^a\beta^b\gamma^c\delta^d$ implies
\[
2>\tfrac{2}{3}a+\tfrac{2}{3}c+\tfrac{2}{5}d.
\]
Similar to the case $\alpha<\beta$, we look for vertices of degree $\ge 4$, such that $2a>b$ and the inequality above is satisfied. The only possible quadruples $(a,b,c,d)$ are 
\[
(2,3,0,1), (2,2,0,1), (2,1,0,1), (2,3,0,0), (2,2,0,0), 
\]
\[
(1,1,1,1), (1,1,0,3), (1,1,0,2), (1,0,0,3). 
\]
By $2\alpha+\beta>\alpha+2\beta=2\pi$, we know all $(2,*,*,*)$ cannot be vertices. By $2\delta=\frac{4}{5}\pi>\frac{2}{3}\pi>\beta$, and $\gamma>\delta$, and $\alpha+2\beta=2\pi$, we know $(1,1,1,1),(1,1,0,3),(1,1,0,2)$ are not vertices. Therefore $(1,0,0,3)$ is the only remaining possibility. If $(1,0,0,3)$ is indeed a vertex, then 
\[
\text{AVC}
\supset\{\alpha\beta^2\gamma\delta\colon \alpha\beta^2,\gamma^3,\alpha\delta^3,\delta^5\}
\supset\{\alpha\beta^2\gamma\delta\colon \alpha\beta^2,\gamma^3,\alpha\delta^3\}.
\]
The AVC on the right has been handled in Section \ref{routine1}. The full AVC is given on page \pageref{eg12}, and the current case is exactly the case $f=60$ in that full AVC. As viewed from $\{\alpha\beta^2\gamma\delta\colon \alpha\beta^2,\gamma^3,\delta^5\}$, the result  is also included in the AVC on page \pageref{eg34}. 

If $(1,0,0,3)$ is not a vertex, then we have $2a=b$ in all the optional vertices. Then by $\alpha+2\beta=2\pi$, a vertex is $\alpha\beta^2,\gamma^c\delta^d$. Then by $\gamma=\tfrac{2}{3}\pi,\delta=\tfrac{2}{5}\pi$, we get $\gamma^c\delta^d=\gamma^3,\delta^5$. We conclude $\alpha\beta^2,\gamma^3,\delta^5$ are all the vertices. Furthermore, we solve the angle counting equations and get the full AVC
\[
\{60\alpha\beta^2\gamma\delta\colon 60\alpha\beta^2,20\gamma^3,12\delta^5\}.
\]
This is included in the AVC on page \pageref{eg34}.

Next we turn to $\{\alpha\beta^2\gamma\delta\colon \alpha\beta^2,\gamma^3,\gamma\delta^3\}$. The angle sums of $\alpha\beta^2,\gamma^3,\gamma\delta^3$ and the angle sum for pentagon $\alpha\beta^2\gamma\delta$ imply
\[
\alpha+2\beta=2\pi,\,
\gamma=\tfrac{2}{3}\pi,\,
\delta=\tfrac{4}{9}\pi,\,
f=36.
\]
If $\alpha<\beta$, then $\alpha<\frac{2}{3}\pi<\beta$, and the angle sum of $\alpha^a\beta^b\gamma^c\delta^d$ implies
\[
2>\tfrac{2}{3}b+\tfrac{2}{3}c+\tfrac{4}{9}d.
\]
This is a tighter constraint than $\{\alpha\beta^2\gamma\delta\colon \alpha\beta^2,\gamma^3,\delta^5\}$. By the similar argument, we find no vertices satisfying $2a<b$. The case of $\alpha>\beta$ can also be handled similarly, and we find no vertices satisfying $2a>b$. Therefore all optional vertices must satisfy $2a=b$, and we find $\alpha\beta^2,\gamma^3,\gamma\delta^3$ are all the vertices. Then we solve the angle counting equations, and get the full AVC on page \pageref{eg32}
\[
\{36\alpha\beta^2\gamma\delta\colon 36\alpha\beta^2,8\gamma^3,12\gamma\delta^3\}.
\]

Finally, for $\{\alpha\beta^2\gamma\delta\colon \alpha\beta^2,\gamma^3,\delta^4\}$, we have
\[
\alpha+2\beta=2\pi,\,
\gamma=\tfrac{2}{3}\pi,\,
\delta=\tfrac{1}{2}\pi,\,
f=24.
\]
By similar argument, we get the full AVC on page \pageref{eg33}
\[
\{24\alpha\beta^2\gamma\delta\colon 24\alpha\beta^2,8\gamma^3,6\delta^4\}.
\]

\medskip

\noindent{\bf Example 2}. $\{\alpha^2\beta\gamma\delta\colon \alpha\beta^2,\alpha^2\gamma\}$, no other degree $3$ vertices.

\medskip

By Lemma \ref{hdeg}, one of $\alpha\delta^3,\beta\delta^3,\gamma\delta^3,\delta^4,\delta^5$ is a vertex. The cases of $\alpha\delta^3$ and $\delta^5$ can be handled by Section \ref{routine1}, and the results are given on pages \pageref{eg21} and \pageref{eg25}. The cases of $\beta\delta^3$ and $\delta^4$ can be handled by Section \ref{routine2}, and the results are given on pages \pageref{eg22} and \pageref{eg24}. 

It remains to study $\{\alpha^2\beta\gamma\delta\colon \alpha\beta^2,\alpha^2\gamma,\gamma\delta^3\}$. The order matrix is still invertible. The angle sums of $\alpha\beta^2,\alpha^2\gamma,\gamma\delta^3$ and the angle sum for pentagon $\alpha^2\beta\gamma\delta$ imply
\[
\alpha=\tfrac{24}{f}\pi,\,
\beta=(1-\tfrac{12}{f})\pi,\,
\gamma=(2-\tfrac{48}{f})\pi,\,
\delta=\tfrac{16}{f}\pi.
\]
The difficulty of the case is that {\em two} angles have zero constant parts. 

By all angles positive, we get $f>24$. By all angles distinct, we get $f\ne 28,32,36$. By $f\ge 26$, we get lower bounds $\beta\ge \tfrac{7}{13}\pi$ and $\gamma\ge \tfrac{1}{13}\pi$. Then 
the angle sum of $\alpha^a\beta^b\gamma^c\delta^d$ implies $\tfrac{7}{13}b+\tfrac{1}{13}c\le 2$. Since there are too many possible choices of $(b,c)$ satisfying the inequality, we treat the first couple of $f$ separately, and then consider those $f$ with bigger lower bound. 

If $f\le 32$, then by $f>24$ and $f\ne 28,32$, we get two cases
\begin{align*}
f=26 &\colon
\alpha=\tfrac{12}{13}\pi,\,
\beta=\tfrac{7}{13}\pi,\,
\gamma=\tfrac{2}{13}\pi,\,
\delta=\tfrac{8}{13}\pi. \\
f=30 &\colon
\alpha=\tfrac{4}{5}\pi,\,
\beta=\tfrac{3}{5}\pi,\,
\gamma=\tfrac{2}{5}\pi,\,
\delta=\tfrac{8}{15}\pi.
\end{align*}
Then it is easy to find the corresponding optional vertices in Table \ref{2abcd||1a2b|2a1c|1c3d_vertex}.

For $f>32$ (and $f\ne 36$), we get lower bounds $\beta>\frac{5}{8}\pi$, $\gamma>\frac{1}{2}\pi$. Then the angle sum of $\alpha^a\beta^b\gamma^c\delta^d$ implies
\[
3a+2d=\tfrac{f}{8}(2-b-2c)+\tfrac{3}{2}(b+4c)
=\lambda f+\mu,\quad
2>\tfrac{5}{8}b+\tfrac{1}{2}c.
\]
Substituting all pairs $(b,c)$ of non-negative integers  satisfying the inequality into the formula for $3a+2d$, we get the following:
\begin{enumerate}
\item $\lambda=0$: This means $(b,c)=(2,0),(0,1)$. Solving the corresponding $3a+2d=3,6$, we get $(a,b,c,d)=(1,2,0,0),(2,0,1,0),(0,0,1,3)$. These correspond to the three existing vertices $\alpha\beta^2,\alpha^2\gamma,\gamma\delta^3$.
\item $\lambda<0$: This means $(b,c)=(2,1),(1,2), (1,1)$, and we get $3a+2d=-\frac{f}{4}+9,-\frac{3f}{2}+\frac{27}{2},-\frac{f}{2}+\frac{15}{2}$, respectively. By $3a+2d\ge 0$, we get upper bounds for $f$. Then for finitely many possible choices of $f>32$ and $f\ne 36$ within the upper bound, we find possible non-negative integer solutions of $a,d$. We end up with only two possibilities:
\begin{align*}
f=40, \, \gamma^2\delta &\colon 
\alpha=\tfrac{3}{5}\pi,\,
\beta=\tfrac{7}{10}\pi,\,
\gamma=\tfrac{4}{5}\pi,\,
\delta=\tfrac{2}{5}\pi; \\
f=44, \, \beta\gamma\delta &\colon 
\alpha=\tfrac{6}{11}\pi,\,
\beta=\tfrac{8}{11}\pi,\,
\gamma=\tfrac{10}{11}\pi,\,
\delta=\tfrac{4}{11}\pi.
\end{align*}
Then it is easy to find the corresponding optional vertices in Table \ref{2abcd||1a2b|2a1c|1c3d_vertex}.
\item $\lambda>0$: This means $(b,c)=(1,0),(0,0)$, and we get two families of optional vertices
\begin{align*}
\alpha^a\beta\delta^d
&\colon 3a+2d=\tfrac{f+12}{8},\, f=4(8); \\
\alpha^{a'}\delta^{d'}
&\colon 3a'+2d'=\tfrac{f}{4},\, f=0(4).
\end{align*}
Considering $f>24$, $f\ne 28,32,36$, that the cases $f=26,30,40,44$ are separately treated, and that $f=0(4)$ for the two vertex families above, this remaining case is only for $f\ge 48$. The vertices for this remaining case are the three existing vertices $\alpha\beta^2,\alpha^2\gamma,\gamma\delta^3$ and the two families above.
\end{enumerate}

\begin{table}[htp]
\centering
\scalebox{1}{
\begin{tabular}{|c |c |c |c |c |c |}
\hline
\multirow{2}{*}{number} & 
\multicolumn{4}{|c|}{vertex $\alpha^a\beta^b\gamma^c\delta^d$} &  
\multirow{2}{*}{condition} \\
\cline{2-5}
 & $a$ & $b$ & $c$ & $d$ &   \\
\hline \hline
$x_1$ & 1 & 2 & 0 & 0 &  \\
\cline{1-5} 
$x_2$ & 2 & 0 & 1 & 0 &  \\
\cline{1-5}
$x_3$ & 0 & 0 & 1 & 3 &  \\
\hline \hline  
$y_1$ & 1 & 0 & 3 & 1 & 
\multirow{7}{*}{$f=26$}  \\
\cline{1-5} 
$y_2$ & 0 & 2 & 2 & 1 &   \\
\cline{1-5} 
$y_3$ & 0 & 0 & 5 & 2 &  \\
\cline{1-5} 
$y_4$ & 1 & 0 & 7 & 0 &   \\
\cline{1-5} 
$y_5$ & 0 & 2 & 6 & 0 &   \\
\cline{1-5} 
$y_6$ & 0 & 0 & 9 & 1 &  \\
\cline{1-5} 
$y_7$ & 0 & 0 & 13 & 0 &  \\
\hline  
$y_1$ & 1 & 0 & 3 & 0 &  
\multirow{3}{*}{$f=30$}  \\
\cline{1-5} 
$y_2$ & 0 & 2 & 2 & 0 &  \\
\cline{1-5} 
$y_3$ & 0 & 0 & 5 & 0 &  \\
\hline  
$y_1$ & 0 & 0 & 2 & 1 & 
\multirow{3}{*}{$f=40$}  \\
\cline{1-5} 
$y_2$ & 2 & 0 & 0 & 2 &  \\
\cline{1-5} 
$y_3$ & 0 & 0 & 0 & 5 &  \\
\hline  
$y_1$ & 0 & 1 & 1 & 1 & 
\multirow{3}{*}{$f=44$}  \\
\cline{1-5} 
$y_2$ & 3 & 0 & 0 & 1 &   \\
\cline{1-5} 
$y_3$ & 1 & 0 & 0 & 4 &  \\
\hline \hline
$z_i$ & $a_i$ & 1 & 0 & $d_i$ &  
$\begin{array}{c} 3a_i+2d_i=\frac{f+12}{8} \\ f=4(8),f\ge 48\end{array}$ \\
\hline 
$z'_i$ & $a_i'$ & 0 & 0 & $d_i'$ &  
$\begin{array}{c} 3a'_i+2d'_i=\frac{f}{4} \\ f=0(4),f\ge 48\end{array}$ \\
\hline 
\end{tabular}
}
\caption{Vertices for $\{\alpha^2\beta\gamma\delta\colon \alpha\beta^2,\alpha^2\gamma,\gamma\delta^3\}$.}
\label{2abcd||1a2b|2a1c|1c3d_vertex}
\end{table}

By solving the angle counting equations, we get contradictions for $f=26,30$ and get full AVCs for $f=40,44$ on page \pageref{eg23}. For $f\ge 48$, the angle counting equations are
\begin{align*}
2f
&= x_1+2x_2+\ssum a_iz_i+\ssum a_i'z_i', \\
f
&= 2x_1+\ssum z_i, \\
f
&= x_2+x_3, \\
f
&= 3x_3+\ssum d_iz_i+\ssum d_i'z_i'.
\end{align*}
The solution is expressions of $x_i$ in terms of $f,z_i,z_i'$
\begin{align*}
x_1&=\tfrac{1}{2}f-\tfrac{1}{2}\ssum z_i, \\
x_2&=\tfrac{2}{3}f+\tfrac{1}{3}\ssum d_iz_i+\tfrac{1}{3}\ssum d'_iz'_i, \\
x_3&=\tfrac{1}{3}f-\tfrac{1}{3}\ssum d_iz_i-\tfrac{1}{3}\ssum d'_iz'_i,
\end{align*}
together with an equality involving only $f,z_i,z_i'$. To get this equality, we may start with the vertex counting equation \eqref{vcountf}
\[
\tfrac{f}{2}-6
=x_3+\ssum(a_i+d_i-2)z_i+\ssum(a_i'+d_i'-3)z_i.
\]
Substituting the formula for $x_3$, we get
\begin{align*}
\tfrac{f}{2}-\tfrac{f}{3}-6
&=\ssum(a_i+d_i-\tfrac{1}{3}d_i-2)z_i+\ssum(a_i'+d_i'-\tfrac{1}{3}d_i'-3)z_i \\
&=\tfrac{1}{24}\ssum(f-36)z_i+\tfrac{1}{12}\ssum(f-36)z_i.
\end{align*}
Canceling $f-36$ on both sides, we get
\[
\ssum z_i+2\ssum z'_i=4.
\]

\medskip

\noindent{\bf Example 3}. 
$\{\alpha^2\beta\gamma\delta\colon \alpha\beta\gamma,\alpha\delta^2\}$.

\medskip

The example is not covered by Proposition \ref{3complete}. According to Table \ref{deg3AVC}, for $\{\alpha^2\beta\gamma\delta\colon \alpha\beta\gamma,\alpha\delta^2\}$, we need to consider the possibility that $\beta^2\delta$ or $\beta^3$ appears as a vertex. Since the order matrix is invertible in these two cases, the routines in Sections \ref{routine1} and \ref{routine2} can be applied, and we omit the details.

Therefore we assume there are no other degree $3$ vertices. The order matrix
\[
(P\, X)=\begin{pmatrix}
2 & 1 & 1 \\
1 & 1 & 0 \\
1 & 1 & 0 \\
1 & 0 & 2 
\end{pmatrix}
\]
is singular due to lack of enough existing vertices. The angle sums of $\alpha\beta\gamma,\alpha\delta^2$ and the angle sum for $\alpha^2\beta\gamma\delta$ imply
\[
\alpha=\tfrac{8}{f}\pi,\,
\beta+\gamma=(2-\tfrac{8}{f})\pi,\,
\delta=(1-\tfrac{4}{f})\pi.
\]

By $f\ge 16$, we get
\[
\beta+\gamma\ge\tfrac{3}{2}\pi,\quad
\delta\ge \tfrac{3}{4}\pi,\quad
\beta+\gamma+\delta=3\delta>2\pi.
\]
Due to the symmetry of exchanging $\beta$ and $\gamma$, we may assume $\beta<\gamma$. Then $\beta<\delta<\gamma$.

By $\gamma>\delta>(1-\tfrac{4}{f})\pi=\frac{3}{4}\pi$, the angle sum of $\alpha^a\beta^b\gamma^c\delta^d$ implies
\[
2\ge \tfrac{3}{4}c+\tfrac{3}{4}d.
\]
Then we have $c+d\le 2$ and the following possible high degree vertices $(a,b,c,d)$:
\begin{enumerate}
\item $c+d=2$: By $\alpha+2\delta=2\pi$ and $\gamma>\delta$, we get $a=0$. Then by high degree, we get $b\ge 2$. Then by $\beta+\gamma+\delta\ge 2\pi$ and $\gamma>\delta$, the only possibility is $(0,b,0,2)$ with $b\ge 2$. 
\item $(c,d)=(0,1)$: We get $(a,b,0,1)$ with $a+b\ge 3$. 
\item $(c,d)=(1,0)$: By $\alpha+\beta+\gamma=2\pi$, either $a$ or $b$ is $0$. Then we get $(a,0,1,0)$ with $a\ge 3$, or $(0,b,1,0)$ with $b\ge 3$.
\item $(c,d)=(0,0)$: We get $(a,b,0,0)$ with $a+b\ge 4$. 
\end{enumerate}
In summary, the possible high degree vertices are the following
\[
(0,b,0,2), \,
(a,b,0,1), \,
(a,0,1,0), \,
(0,b,1,0), \,
(a,b,0,0).
\] 

The pentagon $\alpha^2\beta\gamma\delta$ implies that $\beta,\gamma$ appear the same number of times in the tiling. Since $\beta,\gamma$ appear the same number of times at existing vertices $\alpha\beta\gamma,\alpha\delta^2$, by Lemma \ref{counting}, we know that, among the five families of high degree vertices above, either only those with $b=c$ can appear, or at least one with $b<c$ must appear. 

Suppose only those with $b=c$ can appear. Then the high degree vertices are either $(a,0,0,1)$ with $a\ge 3$, or $(a,0,0,0)$ with $a\ge 4$. The angle sums of the vertices imply $a=\frac{f+4}{8}$ (i.e., $\alpha^{\frac{f+4}{8}}\delta$), or $a=\frac{f}{4}$ (i.e., $\alpha^{\frac{f}{4}}$). Then we solve the angle counting equations for $\{\alpha^2\beta\gamma\delta\colon \alpha\beta\gamma,\alpha\delta^2 \opt \alpha^{\frac{f+4}{8}}\delta,\alpha^{\frac{f}{4}}\}$ and get the full AVC
\begin{align*}
\{f\alpha^2\beta\gamma\delta
&\colon f\alpha\beta\gamma,\,
(\tfrac{1}{2}f-2+y_2)\alpha\delta^2 \opt 
(4-2y_2)\alpha^{\frac{f+4}{8}}\delta,\,
y_2\alpha^{\frac{f}{4}}\},  \\
&\alpha=\tfrac{8}{f}\pi,\,
\beta+\gamma=(2-\tfrac{8}{f})\pi,\,
\delta=(1-\tfrac{4}{f})\pi.   
\end{align*}

Suppose there is a high degree vertex with $b<c$. The only possibility in the list is $(a,0,1,0)$, $a\ge 3$. We use $g$ (another variable in addition to $f$) in place of $a$ and get an updated AVC $\{\alpha^2\beta\gamma\delta\colon \alpha\beta\gamma,\alpha\delta^2,\alpha^g\gamma\}$, in which all vertices are necessary. The updated order matrix
\[
(P\, X')=\begin{pmatrix}
2 & 1 & 1 & g \\
1 & 1 & 0 & 0 \\
1 & 1 & 0 & 1\\
1 & 0 & 2 & 0
\end{pmatrix}
\]
is now invertible. Then we may calculate all the angles
\[
\alpha=\tfrac{8}{f}\pi,\,
\beta=\tfrac{8(g-1)}{f}\pi,\,
\gamma=(2-\tfrac{8g}{f})\pi,\,
\delta=(1-\tfrac{4}{f})\pi.
\]
We have $\alpha<\beta$.

The remaining four families in the list above are the possible optional vertices. In fact, $(0,b,0,2)$ cannot be a vertex because $b\ge 2$ implies $b\beta+2\delta > \alpha+2\delta = 2\pi$. Moreover, $(0,b,1,0)$ is also not a vertex because $b\ge 3$ implies $b\beta+\gamma> 2\beta+\gamma > \alpha+\beta+\gamma=2\pi$. On the other hand, the angle sums of $(a,b,0,1)$ and $(a',b',0,0)$ imply
\begin{align*}
(a,b,0,1)
&\colon a+(g-1)b=\tfrac{f+4}{8},\, f=4(8). \\
(a',b',0,0)
&\colon a'+(g-1)b'=\tfrac{f}{4},\, f=0(4).
\end{align*}
Then we solve the angle counting equations similar to Example 2 of this section and get the full AVC 
\begin{align*}
\{
f\alpha^2\beta\gamma\delta &\colon 
x_1\alpha\beta\gamma,\,
x_2\alpha\delta^2,\,
x_3\alpha^g\gamma \opt 
y_i\alpha^{a_i}\beta^{b_i}\delta, \,
y_i'\alpha^{a'_i}\beta^{b'_i}
\},  \\
& 
a_i+(g-1)d_i=\tfrac{f+4}{8}, \,
a_i'+(g-1)d_i'=\tfrac{f}{4}, \\
& x_1 =f-\ssum b_iy_i-\ssum b_i'y_i',   \\
& x_2 =\tfrac{1}{2}f-\tfrac{1}{2}\ssum y_i,   \\
& x_3 =\ssum b_iy_i+\ssum b_i'y_i',   \\
& \ssum y_i+2\ssum y'_i=4, \\
&
\alpha=\tfrac{8}{f}\pi,\,
\beta=\tfrac{8(g-1)}{f}\pi,\,
\gamma=(2-\tfrac{8g}{f})\pi,\,
\delta=(1-\tfrac{4}{f})\pi.   
\end{align*}

\medskip

\noindent{\bf Example 4}. 
$\{\alpha\beta\gamma\delta\epsilon\colon \alpha\beta\gamma,\delta^3\}$, no other degree $3$ vertices.

\medskip

The example is not covered by Proposition \ref{3complete}, and the assumption of no other degree $3$ vertices follows from Table \ref{deg3AVC}. By Lemma \ref{hdeg} and up to the symmetry of exchanging $\alpha,\beta,\gamma$, we may assume that one of $\alpha\epsilon^3,\delta\epsilon^3,\epsilon^4,\epsilon^5$ is a vertex. 

Suppose $\alpha\epsilon^3$ is a vertex. The angle sums of $\alpha\beta\gamma,\delta^3,\alpha\epsilon^3$ and the angle sum for $\alpha\beta\gamma\delta\epsilon$ imply
\[
\alpha=(1-\tfrac{12}{f})\pi,\,
\beta+\gamma=(1+\tfrac{12}{f})\pi,\,
\delta=\tfrac{2}{3}\pi,\,
\epsilon=(\tfrac{1}{3}+\tfrac{4}{f})\pi.
\]
Using the positive lower bounds for $\alpha,\delta,\epsilon$, we may find optional high order vertices similar to Example 1 of this section. We omit the details.

Suppose $\epsilon^5$ is a vertex. The angle sums of $\alpha\beta\gamma,\delta^3,\epsilon^5$ and the angle sum for $\alpha\beta\gamma\delta\epsilon$ imply
\[
\alpha+\beta+\gamma=2\pi,\,
\delta=\tfrac{2}{3}\pi,\,
\epsilon=\tfrac{2}{5}\pi,\,
f=60.
\]
This is similar to $\{\alpha\beta^2\gamma\delta\colon \alpha\beta^2,\gamma^3,\delta^5\}$ in Example 1 of this section, except $\alpha+2\beta=2\pi$ is changed to $\alpha+\beta+\gamma=2\pi$. It is easy to see that $\delta^d\epsilon^e=\delta^3,\epsilon^5$. Therefore if $\alpha\beta\gamma$ is the only vertex involving the three angles, then $\alpha\beta\gamma,\delta^3,\epsilon^5$ are all the vertices, and we can easily get the full AVC
\begin{equation}\label{60a}
\{
60\alpha\beta\gamma\delta\epsilon \colon 
60\alpha\beta\gamma,\,
20\delta^3,\,
12\epsilon^5 \}.
\end{equation}

Next, we assume $\alpha,\beta,\gamma$ appear at vertices other than $\alpha\beta\gamma$. By applying Lemma \ref{counting} to any pair from $\alpha,\beta,\gamma$, we know each of them appears at some high degree vertices. By symmetry, we may further assume $\alpha<\beta<\gamma$. Then by $\alpha+\beta+\gamma=2\pi$ and $\delta=\frac{2}{3}\pi$, we know the angle sums of $\alpha\gamma^2\cdots,\beta\gamma^2\cdots,\beta^2\gamma\cdots,\beta\gamma\delta\cdots$ are $>2\pi$ and therefore cannot be vertices.

If $\gamma\cdots$ always has $\alpha$, then by no $\alpha\gamma^2\cdots$, and applying Lemma \ref{counting} to $\alpha,\gamma$, we get $\alpha\cdots=\gamma\cdots=\alpha\beta\gamma,\alpha\gamma\delta^d\epsilon^e$, with $d+e\ge 2$. With the exception of $(d,e)=(0,2)$, the angle sum of $\alpha\gamma\delta^d\epsilon^e$ implies $\beta=d\delta+e\epsilon\ge \pi$, contradicting $\alpha<\beta<\gamma$ and $\alpha+\beta+\gamma=2\pi$. Therefore the only high degree vertex $\gamma\cdots$ is $\alpha\gamma\epsilon^2$. The angle sum of $\alpha\gamma\epsilon^2$ further implies
\[
\alpha+\gamma=\tfrac{6}{5}\pi,\,
\beta=\tfrac{4}{5}\pi,\,
\delta=\tfrac{2}{3}\pi,\,
\epsilon=\tfrac{2}{5}\pi.
\]
Then we find $\beta^b\epsilon^e=\beta\epsilon^3$. We conclude $\alpha\beta\gamma,\delta^3,\epsilon^5,\alpha\gamma\epsilon^2,\beta\epsilon^3$ are all the vertices. Then we solve the angle counting equations and get the full AVC
\[
\{
60\alpha\beta\gamma\delta\epsilon \colon 
(60-y_1)\alpha\beta\gamma,\,
20\delta^3,\,
(12-y_1)\epsilon^5,\,
y_1\alpha\gamma\epsilon^2\opt 
y_1\beta\epsilon^3\}.
\]
Note that $y_1=0$ reduces to a special case (due to $\beta=\frac{4}{5}\pi$) of \eqref{60a}. Moreover, $y_1>0$ belongs to the case $\{\alpha\beta\gamma\delta\epsilon \colon \alpha\beta\gamma,\delta^3,\alpha\epsilon^3 \}$ after exchanging $\alpha$ and $\beta$.

If there is a vertex $\gamma\cdots$ without $\alpha$, then the vertex has high degree. By the similar angle sum consideration, the vertex is $\gamma\epsilon^3$. After exchanging $\alpha$ and $\gamma$, we are in the case $\{\alpha\beta\gamma\delta\epsilon \colon \alpha\beta\gamma,\delta^3,\alpha\epsilon^3 \}$ .

We conclude that the full AVC derived from $\{\alpha\beta\gamma\delta\epsilon\colon \alpha\beta\gamma,\delta^3,\epsilon^5\}$ is either \eqref{60a}, or belongs to the full AVC derived from $\{\alpha\beta\gamma\delta\epsilon \colon \alpha\beta\gamma,\delta^3,\alpha\epsilon^3 \}$ after some exchange among $\alpha,\beta,\gamma$.

Suppose $\delta\epsilon^3$ is a vertex. The angle sums of $\alpha\beta\gamma,\delta^3,\delta\epsilon^3$ and the angle sum for $\alpha\beta\gamma\delta\epsilon$ imply
\[
\alpha+\beta+\gamma=2\pi,\,
\delta=\tfrac{2}{3}\pi,\,
\epsilon=\tfrac{4}{9}\pi,\,
f=36.
\]
Like Example 1 in this section, this is tighter than the case $f=60$. In fact, the similar argument shows that, if $\alpha<\beta<\gamma$, then $\gamma\cdots=\alpha\beta\gamma,\alpha\gamma\epsilon^2$. By applying Lemma \ref{counting} to $\alpha,\gamma$, this implies $\alpha\cdots=\alpha\beta\gamma,\alpha\gamma\epsilon^2$. 

The angle sum of $\alpha\gamma\epsilon^2$ further implies 
\[
\alpha+\gamma=\tfrac{10}{9}\pi,\,
\beta=\tfrac{8}{9}\pi,\,
\delta=\tfrac{2}{3}\pi,\,
\epsilon=\tfrac{4}{9}\pi.
\]
By $\alpha\cdots=\gamma\cdots=\alpha\beta\gamma,\alpha\gamma\epsilon^2$, we know $\beta\cdots=\alpha\beta\gamma$ or $\beta\cdots$ has no $\alpha,\gamma$. By the angle values, we get $\beta\cdots=\alpha\beta\gamma,\beta\delta\epsilon$. Since $\alpha\beta\gamma,\delta^3$ are all the degree $3$ vertices, we get $\beta\cdots=\alpha\beta\gamma$. Then applying Lemma \ref{counting} to $\beta,\gamma$, we get a contradiction. 

Therefore $\alpha\gamma\epsilon^2$ is not a vertex, and $\alpha\beta\gamma,\delta^3,\delta\epsilon^3$ are all the vertices. Then we may further calculate the numbers of vertices and get the full AVC
\begin{equation}\label{36a}
\{
36\alpha\beta\gamma\delta\epsilon \colon 
36\alpha\beta\gamma,\,
8\delta^3,\,
12\delta\epsilon^3 \}.
\end{equation}

Finally, suppose $\epsilon^4$ is a vertex. The angle sums of $\alpha\beta\gamma,\delta^3,\epsilon^4$ and the angle sum for $\alpha\beta\gamma\delta\epsilon$ imply
\[
\alpha+\beta+\gamma=2\pi,\,
\delta=\tfrac{2}{3}\pi,\,
\epsilon=\tfrac{1}{2}\pi,\,
f=24.
\]
Similar to the cases $f=60,36$, we find $\alpha\beta\gamma,\delta^3,\epsilon^4$ are the only vertices, and the full AVC is
\begin{equation}\label{24a}
\{
24\alpha\beta\gamma\delta\epsilon \colon 
24\alpha\beta\gamma,\,
8\delta^3,\,
6\epsilon^4 \}.
\end{equation}

\section{Classification of AVCs}
\label{classify}

\subsection{Up to Three Distinct Angles at Degree $3$ Vertices}
\label{classify3}

Starting from Table \ref{3completeAVC} and applying the techniques in Sections \ref{routine1}, \ref{routine2}, \ref{routine3}, we get the complete list of all the full AVCs with up to three distinct angles at degree $3$ vertices. One should keep the following in mind while reading the subsequent list.
\begin{enumerate}
\item The list assumes $f\ge 16$ (equivalent to $f\ne 12$). It contains all AVCs with up to three distinct angles, and all AVCs with four distinct angles but one angle not appearing at degree $3$ vertices.
\item All coefficients in the necessary part (i.e., before the divider $\opt$) are positive integers. All coefficients in the optional part (i.e., after the divider $\opt$) are non-negative integers. 
\item All angles are distinct. This implies that $f$ should not take certain specific values. Although allowing angles to be equal can still give an AVC, such an AVC is included in another AVC in the list.
\item All optional vertices should have non-negative integer exponents and have degree $\ge 4$. This implies that, for the corresponding coefficient to be nonzero, $f$ should satisfy some modulus condition and has some lower bound. 
\end{enumerate}

The complete list is headed by degree $3$ vertices listed in Table \ref{3completeAVC}, and a vertex of degree $4$ or $5$ in case there is an extra angle not appearing at degree $3$ vertices.

\begin{itemize}
\item $\{\alpha^3\}$: The AVCs from $\{\alpha^3\}$ always include an extra angle $\beta$ not appearing at degree $3$ vertices.
\begin{itemize}
\item 
$\{
24\alpha^4\beta \colon 
32\alpha^3,\,
6\beta^4
\}, \,
\alpha=\tfrac{2}{3}\pi,\,
\beta=\tfrac{1}{2}\pi$.
\item 
$\{
36\alpha^4\beta \colon 
44\alpha^3,\,
12\alpha\beta^3
\}, \,
\alpha=\tfrac{2}{3}\pi,\,
\beta=\tfrac{4}{9}\pi$.
\item
$\{
60\alpha^4\beta \colon 
80\alpha^3,\,
12\beta^5
\}, \,
\alpha=\tfrac{2}{3}\pi,\,
\beta=\tfrac{2}{5}\pi$.
\end{itemize}
\item $\{\alpha\beta^2\}$: \label{eg28}
\begin{itemize}
\item 
$\{
f\alpha^2\beta^3 \colon 
(\tfrac{3}{2}f-2+y_2)\alpha\beta^2 \opt 
(4-2y_2)\alpha^{\frac{f+4}{8}}\beta,\,
y_2\alpha^{\frac{f}{4}}
\}, 
\\
\alpha=\tfrac{8}{f}\pi,\,
\beta=(1-\tfrac{4}{f})\pi$.
\end{itemize}
\item $\{\alpha\beta\gamma,\alpha^3\}$:
\begin{itemize}
\item 
$\{
24\alpha^2\beta^2\gamma \colon 
24\alpha\beta\gamma,\,
8\alpha^3,\,
6\beta^4 
\}, \,
\alpha=\tfrac{2}{3}\pi,\,
\beta=\tfrac{1}{2}\pi,\,
\gamma=\tfrac{5}{6}\pi$.
\item 
$\{
36\alpha^2\beta^2\gamma \colon 
36\alpha\beta\gamma,\,
8\alpha^3,\,
12\alpha\beta^3
\}, \,
\alpha=\tfrac{2}{3}\pi,\,
\beta=\tfrac{4}{9}\pi,\,
\gamma=\tfrac{8}{9}\pi$.
\item 
$\{
60\alpha^2\beta^2\gamma \colon 
60\alpha\beta\gamma,\,
20\alpha^3,\,
12\beta^5
\}, \,
\alpha=\tfrac{2}{3}\pi,\,
\beta=\tfrac{2}{5}\pi,\,
\gamma=\tfrac{14}{15}\pi$. 
\end{itemize}
\item $\{\alpha\beta\gamma,\alpha^3,\alpha\delta^3\}$:
\begin{itemize}
\item 
$\{
36\alpha^2\beta\gamma\delta \colon 
36\alpha\beta\gamma,\,
8\alpha^3,\,
12\alpha\delta^3 
\}, \,
\alpha=\tfrac{2}{3}\pi,\,
\beta+\gamma=\tfrac{4}{3}\pi,\,
\gamma=\tfrac{4}{9}\pi$. 
\end{itemize}
\item $\{\alpha\beta\gamma,\alpha^3,\beta\delta^3\}$:\label{eg11}
\begin{itemize}
\item 
$\{
48\alpha^2\beta\gamma\delta \colon 
36\alpha\beta\gamma,\,
20\alpha^3,\,
12\beta\delta^3 \opt
6\gamma^2\delta^2
\},
\\
\alpha=\tfrac{2}{3}\pi,\,
\beta=\tfrac{3}{4}\pi,\,
\gamma=\tfrac{7}{12}\pi,\,
\delta=\tfrac{5}{12}\pi$.
\item
$\{
60\alpha^2\beta\gamma\delta \colon 
(60-y_1-3y_2)\alpha\beta\gamma,\,
(20+y_2)\alpha^3,\,
(y_1+3y_2)\beta\delta^3 
\\
\opt 
y_1\alpha\gamma\delta^2,\,
y_2\gamma^3\delta,\,
(12-y_1-2y_2)\delta^5
\}, 
\\
\alpha=\tfrac{2}{3}\pi,\,
\beta=\tfrac{4}{5}\pi,\,
\gamma=\tfrac{18}{15}\pi,\,
\delta=\tfrac{2}{5}\pi$.
\item
$\{
72\alpha^2\beta\gamma\delta \colon 
48\alpha\beta\gamma,\,
32\alpha^3,\,
24\beta\delta^3 \opt
6\gamma^4
\}, 
\\
\alpha=\tfrac{2}{3}\pi,\,
\beta=\tfrac{5}{6}\pi,\,
\gamma=\tfrac{1}{2}\pi,\,
\delta=\tfrac{7}{18}\pi$.
\item
$\{
84\alpha^2\beta\gamma\delta \colon 
(72-y_1)\alpha\beta\gamma,\,
32\alpha^3,\,
(12+y_1)\beta\delta^3 \opt
y_1\alpha\gamma^3,\,
(12-y_1)\gamma^2\delta^3
\},
\\
\alpha=\tfrac{2}{3}\pi,\,
\beta=\tfrac{6}{7}\pi,\,
\gamma=\tfrac{10}{21}\pi,\,
\delta=\tfrac{8}{21}\pi$.
\item
$\{
108\alpha^2\beta\gamma\delta \colon 
(84-y_1)\alpha\beta\gamma,\,
44\alpha^3,\,
(24+y_1)\beta\delta^3 
\\
\opt
y_1\alpha\gamma^2\delta,\,
(12-y_1)\gamma\delta^4
\}, 
\\
\alpha=\tfrac{2}{3}\pi,\,
\beta=\tfrac{8}{9}\pi,\,
\gamma=\tfrac{4}{9}\pi,\,
\delta=\tfrac{10}{27}\pi$.
\item
$\{
132\alpha^2\beta\gamma\delta \colon 
96\alpha\beta\gamma,\,
56\alpha^3,\,
36\beta\delta^3 \opt
12\gamma^3\delta^2
\}, 
\\
\alpha=\tfrac{2}{3}\pi,\,
\beta=\tfrac{10}{11}\pi,\,
\gamma=\tfrac{14}{33}\pi,\,
\delta=\tfrac{4}{11}\pi$.
\item
$\{
156\alpha^2\beta\gamma\delta \colon 
108\alpha\beta\gamma,\,
68\alpha^3,\,
48\beta\delta^3 \opt
12\gamma^4\delta
\}, 
\\
\alpha=\tfrac{2}{3}\pi,\,
\beta=\tfrac{12}{13}\pi,\,
\gamma=\tfrac{16}{39}\pi,\,
\delta=\tfrac{14}{39}\pi$.
\item
$\{
180\alpha^2\beta\gamma\delta \colon 
120\alpha\beta\gamma,\,
80\alpha^3,\,
60\beta\delta^3 \opt
12\gamma^5
\}, 
\\
\alpha=\tfrac{2}{3}\pi,\,
\beta=\tfrac{14}{15}\pi,\,
\gamma=\tfrac{2}{5}\pi,\,
\delta=\tfrac{16}{45}\pi$.
\end{itemize}
\item $\{\alpha\beta\gamma,\alpha^3,\delta^4\}$:
\begin{itemize}
\item 
$\{
24\alpha^2\beta\gamma\delta \colon 
24\alpha\beta\gamma,\,
8\alpha^3,\,
6\delta^4
\}, \,
\alpha=\tfrac{2}{3}\pi,\,
\beta+\gamma=\tfrac{4}{3}\pi,\,
\delta=\tfrac{1}{2}\pi$. 
\end{itemize}
\item $\{\alpha\beta\gamma,\alpha^3,\delta^5\}$:
\begin{itemize}
\item 
$\{
60\alpha^2\beta\gamma\delta \colon 
60\alpha\beta\gamma,\,
20\alpha^3,
12\delta^5
\}, \,
\alpha=\tfrac{2}{3}\pi,\,
\beta+\gamma=\tfrac{4}{3}\pi,\,
\delta=\tfrac{2}{5}\pi$.
\end{itemize}
\item $\{\alpha\beta^2,\alpha^2\gamma\}$: \label{eg27}
\begin{itemize}
\item
$\{
f\alpha^3\beta\gamma \colon 
(\tfrac{1}{2}f-2+y_2)\alpha\beta^2,\,
f\alpha^2\gamma \opt 
(4-2y_2)\alpha^{\frac{f+4}{8}}\beta,\,
y_2\alpha^{\frac{f}{4}}
\}, 
\\
\alpha=\tfrac{8}{f}\pi,\,
\beta=(1-\tfrac{4}{f})\pi,\,
\gamma=(2-\tfrac{16}{f})\pi$.
\item
$\{
f\alpha^2\beta^2\gamma \colon 
x_1\alpha\beta^2,\,
x_2\alpha^2\gamma 
\\
\opt 
y_1\alpha\beta\gamma^{\frac{f+4}{16}},\,
y_2\beta^3\gamma^{\frac{f-12}{16}},\,
y_3\alpha\gamma^{\frac{f+4}{8}},\,
y_4\beta^2\gamma^{\frac{f-4}{8}},\,
y_5\beta\gamma^{\frac{3f-4}{16}},\,
y_6\gamma^{\frac{f}{4}}
\},  
\\
x_1=f
-\tfrac{1}{2}y_1
-\tfrac{3}{2}y_2
-y_4
-\tfrac{1}{2}y_5, 
\\
x_2=\tfrac{1}{2}f
-\tfrac{1}{4}y_1
+\tfrac{3}{4}y_2
-\tfrac{1}{2}y_3
+\tfrac{1}{2}y_4
+\tfrac{1}{4}y_5, 
\\
y_1+y_2+2y_3+2y_4+3y_5+4y_6=8, 
\\
\alpha=(1-\tfrac{4}{f})\pi,\,
\beta=(\tfrac{1}{2}+\tfrac{2}{f})\pi,\,
\gamma=\tfrac{8}{f}\pi$. 
\item 
$\{
32\alpha^2\beta\gamma^2 \colon 
16\alpha\beta^2,\,
24\alpha^2\gamma \opt 
10\gamma^4
\}, \,
\alpha=\tfrac{3}{4}\pi,\,
\beta=\tfrac{5}{8}\pi,\,
\gamma=\tfrac{1}{2}\pi$. 
\item
$\{
52\alpha^2\beta\gamma^2 \colon 
16\alpha\beta^2,\,
44\alpha^2\gamma \opt 
20\beta\gamma^3
\}, \,
\alpha=\tfrac{10}{13}\pi,\,
\beta=\tfrac{8}{13}\pi,\,
\gamma=\tfrac{6}{13}\pi$. 
\end{itemize}
\item $\{\alpha\beta^2, \alpha^2\gamma,\alpha\delta^3\}$: \label{eg21}
\begin{itemize}
\item 
$\{
56\alpha^2\beta\gamma\delta \colon 
28\alpha\beta^2,\,
36\alpha^2\gamma,\,
12\alpha\delta^3 \opt
10\gamma^2\delta^2
\}, 
\\
\alpha=\tfrac{22}{35}\pi,\,
\beta=\tfrac{24}{35}\pi,\,
\gamma=\tfrac{26}{35}\pi,\,
\delta=\tfrac{16}{35}\pi$.
\item
$\{
76\alpha^2\beta\gamma\delta \colon 
(28+y_2)\alpha\beta^2,\,
(56-y_2)\alpha^2\gamma,\,
(12+y_2)\alpha\delta^3 
\\
\opt 
(20-2y_2)\beta\gamma\delta^2,\,
y_2\gamma^3\delta
\}
\\
\alpha=\tfrac{14}{19}\pi,\,
\beta=\tfrac{12}{19}\pi,\,
\gamma=\tfrac{10}{19}\pi,\,
\delta=\tfrac{8}{19}\pi$.
\item
$\{
96\alpha^2\beta\gamma\delta \colon 
48\alpha\beta^2,\,
56\alpha^2\gamma,\,
32\alpha\delta^3 \opt
10\gamma^4
\}, 
\\
\alpha=\tfrac{3}{4}\pi,\,
\beta=\tfrac{5}{8}\pi,\,
\gamma=\tfrac{1}{2}\pi,\,
\delta=\tfrac{5}{12}\pi$.
\item
$\{
116\alpha^2\beta\gamma\delta \colon 
48\alpha\beta^2,\,
76\alpha^2\gamma,\,
32\alpha\delta^3 \opt
20\beta\gamma^2\delta
\}, 
\\
\alpha=\tfrac{22}{29}\pi,\,
\beta=\tfrac{18}{29}\pi,\,
\gamma=\tfrac{14}{29}\pi,\,
\delta=\tfrac{12}{29}\pi$.
\item
$\{
156\alpha^2\beta\gamma\delta \colon 
68\alpha\beta^2,\,
96\alpha^2\gamma,\,
52\alpha\delta^3 \opt
20\beta\gamma^3
\}, 
\\
\alpha=\tfrac{10}{13}\pi,\,
\beta=\tfrac{8}{13}\pi,\,
\gamma=\tfrac{6}{13}\pi,\,
\delta=\tfrac{16}{39}\pi$.
\end{itemize}
\item $\{\alpha\beta^2, \alpha^2\gamma,\beta\delta^3\}$: \label{eg22}
\begin{itemize}
\item 
$\{
f\alpha^2\beta\gamma\delta \colon 
x_1\alpha\beta^2, \,
x_2\alpha^2\gamma, \,
x_3\beta\delta^3 
\\
\opt 
y_1\beta\gamma^{\frac{f-4}{48}}\delta^2, \,
y_2\alpha\gamma^{\frac{f+28}{48}}\delta, \,
y_3\beta^2\gamma^{\frac{f-20}{48}}\delta, \,
y_4\alpha\beta\gamma^{\frac{f+12}{48}}, \,
\\
{}\quad 
y_5\gamma^{\frac{f+12}{48}}\delta^3, 
y_6\beta^3\gamma^{\frac{f-36}{48}}, \,
y_7\beta\gamma^{\frac{f-4}{24}}\delta, \,
y_8\alpha\gamma^{\frac{f+12}{24}}, 
\\
{}\quad 
y_9\beta^2\gamma^{\frac{f-12}{24}}, \,
y_{10}\gamma^{\frac{f+4}{24}}\delta^2, \,
y_{11}\beta\gamma^{\frac{f-4}{16}}, \,
y_{12}\gamma^{\frac{f}{12}}
\},
\\
x_1 
=\tfrac{1}{3}f
-\tfrac{1}{6}y_1
+\tfrac{1}{6}y_2
-\tfrac{5}{6}y_3
-\tfrac{1}{2}y_4
+\tfrac{1}{2}y_5
-\tfrac{3}{2}y_6 
\\ 
{}\qquad -\tfrac{1}{3}y_7
-y_9
+\tfrac{1}{3}y_{10}
-\tfrac{1}{2}y_{11}, 
\\
x_2 
=\tfrac{5}{6}f
+\tfrac{1}{12}y_1
-\tfrac{7}{12}y_2
+\tfrac{5}{12}y_3
-\tfrac{1}{4}y_4
-\tfrac{1}{4}y_5
+\tfrac{3}{4}y_6 
\\
{}\qquad+\tfrac{1}{6}y_7
-\tfrac{1}{2}y_8
+\tfrac{1}{12}y_9
-\tfrac{1}{6}y_{10}
+\tfrac{1}{4}y_{11},
\\
x_3 
=\tfrac{1}{3}f
-\tfrac{2}{3}y_1
-\tfrac{1}{3}y_2
-\tfrac{1}{3}y_3
-y_5
-\tfrac{1}{3}y_7
-\tfrac{2}{3}y_{10}, 
\\
y_1+\cdots+y_6+2(y_7+\cdots+y_{10})+3y_{11}+4y_{12}=8, 
\\
\alpha=(1-\tfrac{12}{f})\pi,\,
\beta=(\tfrac{1}{2}+\tfrac{6}{f})\pi,\,
\gamma=\tfrac{24}{f}\pi,\,
\delta=(\tfrac{1}{2}-\tfrac{2}{f})\pi$.
\end{itemize}
\item $\{\alpha\beta^2, \alpha^2\gamma,\gamma\delta^3\}$: \label{eg23}
\begin{itemize}
\item 
$\{
40\alpha^2\beta\gamma\delta \colon 
20\alpha\beta^2,\,
(30-y_2)\alpha^2\gamma,\,
(14-y_2-2y_3)\gamma\delta^3 
\\
\opt 
(y_2+y_3-2)\gamma^2\delta,\,
y_2\alpha^2\delta^2,\,
y_3\delta^5
\},
\\
\alpha=\tfrac{3}{5}\pi,\,
\beta=\tfrac{7}{10}\pi,\,
\gamma=\tfrac{4}{5}\pi,\,
\delta=\tfrac{2}{5}\pi$.
\item
$\{
44\alpha^2\beta\gamma\delta \colon 
(24-y_2-y_3)\alpha\beta^2, \,
(32-y_2)\alpha^2\gamma, \,
(16-y_2-2y_3)\gamma\delta^3 
\\
\opt 
(2y_2+2y_3-4)\beta\gamma\delta, \,
y_2\alpha^3\delta, \,
y_3\alpha\delta^4
\},
\\
\alpha=\tfrac{6}{11}\pi,\,
\beta=\tfrac{8}{11}\pi,\,
\gamma=\tfrac{10}{11}\pi,\,
\delta=\tfrac{4}{11}\pi$.
\item
$\{
f\alpha^2\beta\gamma\delta \colon 
x_1\alpha\beta^2, \,
x_2\alpha^2\gamma, \,
x_3\gamma\delta^3 \opt 
y_i\alpha^{a_i}\beta\delta^{d_i}, \,
y_i'\alpha^{a_i'}\delta^{d_i'}
\}, 
\\ 
3a_i+2d_i=\tfrac{f+12}{8}, \,
3a_i'+2d_i'=\tfrac{f}{4},\,
f\ge 48,
\\
x_1=\tfrac{1}{2}f-\tfrac{1}{2}\ssum y_i, 
\\
x_2=\tfrac{2}{3}f+\tfrac{1}{3}\ssum d_iy_i+\tfrac{1}{3}\ssum d'_iy'_i, 
\\
x_3=\tfrac{1}{3}f-\tfrac{1}{3}\ssum d_iy_i-\tfrac{1}{3}\ssum d'_iy'_i, 
\\ 
\ssum y_i+2\ssum y'_i=4,
\\
\alpha=\tfrac{24}{f}\pi,\,
\beta=(1-\tfrac{12}{f})\pi,\,
\gamma=(2-\tfrac{48}{f})\pi,\,
\delta=\tfrac{16}{f}\pi$.
\end{itemize}
\item $\{\alpha\beta^2, \alpha^2\gamma,\delta^4\}$: \label{eg24}
\begin{itemize}
\item 
$\{
f\alpha^2\beta\gamma\delta \colon 
x_1\alpha\beta^2, \,
x_2\alpha^2\gamma, \,
x_3\delta^4 
\\
\opt 
y_1\gamma^{\frac{f}{32}}\delta^3, \,
y_2\beta\gamma^{\frac{f-8}{32}}\delta^2, \,
y_3\alpha\gamma^{\frac{f+16}{32}}\delta, \,
y_4\beta^2\gamma^{\frac{f-16}{32}}\delta, 
\\
{}\quad 
y_5\alpha\beta\gamma^{\frac{f+8}{32}}, \,
y_6\beta^3\gamma^{\frac{f-24}{32}}, \,
y_7\gamma^{\frac{f}{16}}\delta^2, \,
y_8\beta\gamma^{\frac{f-4}{16}}\delta, 
\\
{}\quad 
y_9\alpha\gamma^{\frac{f+8}{16}}, \,
y_{10}\beta^2\gamma^{\frac{f-8}{16}}, \,
y_{11}\gamma^{\frac{3f}{32}}\delta, \,
y_{12}\beta\gamma^{\frac{3f-8}{32}}, \,
y_{13}\gamma^{\frac{f}{8}}
\}, 
\\
x_1 
=\tfrac{1}{2}f
-\tfrac{1}{2}y_2
-y_4
-\tfrac{1}{2}y_5
-\tfrac{3}{2}y_6
-\tfrac{1}{2}y_8
-y_{10}
-\tfrac{1}{2}y_{12},
\\
x_2 
=\tfrac{3}{4}f
+\tfrac{1}{4}y_2
-\tfrac{1}{2}y_3
+\tfrac{1}{2}y_4
-\tfrac{1}{4}y_5
+\tfrac{3}{4}y_6
+\tfrac{1}{4}y_8
-\tfrac{1}{2}y_9
+\tfrac{1}{2}y_{10}
+\tfrac{1}{4}y_{12},
\\ 
x_3 
=\tfrac{1}{4}f
-\tfrac{3}{4}y_1
-\tfrac{1}{2}y_2
-\tfrac{1}{4}y_3
-\tfrac{1}{4}y_4
-\tfrac{1}{2}y_7
-\tfrac{1}{4}y_8
-\tfrac{3}{4}y_{11}, 
\\ 
y_1+\cdots+y_6+2(y_7+\cdots+y_{10})+3(y_{11}+y_{12})+4y_{13}=8, 
\\
\alpha=(1-\tfrac{8}{f})\pi,\,
\beta=(\tfrac{1}{2}+\tfrac{4}{f})\pi,\,
\gamma=\tfrac{16}{f}\pi,\,
\delta=\tfrac{1}{2}\pi$.
\end{itemize}
\item $\{\alpha\beta^2, \alpha^2\gamma,\delta^5\}$: \label{eg25}
\begin{itemize}
\item 
$\{
40\alpha^2\beta\gamma\delta \colon 
20\alpha\beta^2,\,
(30-y_1)\alpha^2\gamma,\,
(2-y_1)\delta^5 \opt
y_1\alpha^2\delta^2,\,
(10+y_1)\gamma\delta^3
\}, 
\\
\alpha=\tfrac{3}{5}\pi,\,
\beta=\tfrac{7}{10}\pi,\,
\gamma=\tfrac{4}{5}\pi,\,
\delta=\tfrac{2}{5}\pi$. 
\\
$\{
80\alpha^2\beta\gamma\delta \colon 
40\alpha\beta^2,\,
60\alpha^2\gamma,\,
12\delta^5 \opt
10\gamma^2\delta^2
\}, 
\\
\alpha=\tfrac{7}{10}\pi,\,
\beta=\tfrac{13}{20}\pi,\,
\gamma=\tfrac{3}{5}\pi,\,
\delta=\tfrac{2}{5}\pi$. 
\\
$\{
100\alpha^2\beta\gamma\delta \colon 
40\alpha\beta^2,\,
80\alpha^2\gamma,\,
12\delta^5 \opt
20\beta\gamma\delta^2
\}, 
\\
\alpha=\tfrac{18}{25}\pi,\,
\beta=\tfrac{16}{25}\pi,\,
\gamma=\tfrac{14}{25}\pi,\,
\delta=\tfrac{2}{5}\pi$. 
\\
$\{
120\alpha^2\beta\gamma\delta \colon 
60\alpha\beta^2,\,
90\alpha^2\gamma,\,
22\delta^5 \opt
10\gamma^3\delta
\}, 
\\
\alpha=\tfrac{11}{15}\pi,\,
\beta=\tfrac{19}{30}\pi,\,
\gamma=\tfrac{8}{15}\pi,\,
\delta=\tfrac{2}{5}\pi$. 
\\
$\{
160\alpha^2\beta\gamma\delta \colon 
80\alpha\beta^2,\,
120\alpha^2\gamma,\,
32\delta^5 \opt
10\gamma^4
\}, 
\\
\alpha=\tfrac{3}{4}\pi,\,
\beta=\tfrac{5}{8}\pi,\,
\gamma=\tfrac{1}{2}\pi,\,
\delta=\tfrac{2}{5}\pi$. 
\\
$\{
180\alpha^2\beta\gamma\delta \colon 
80\alpha\beta^2,\,
140\alpha^2\gamma,\,
32\delta^5 \opt
20\beta\gamma^2\delta
\}, 
\\
\alpha=\tfrac{34}{45}\pi,\,
\beta=\tfrac{28}{45}\pi,\,
\gamma=\tfrac{22}{45}\pi,\,
\delta=\tfrac{2}{5}\pi$. 
\\
$\{
260\alpha^2\beta\gamma\delta \colon 
120\alpha\beta^2,\,
200\alpha^2\gamma,\,
52\delta^5 \opt
20\beta\gamma^3
\}, 
\\
\alpha=\tfrac{10}{13}\pi,\,
\beta=\tfrac{8}{13}\pi,\,
\gamma=\tfrac{6}{13}\pi,\,
\delta=\tfrac{2}{5}\pi$.
\end{itemize}
\item $\{\alpha\beta^2,\gamma^3\}$: \label{eg29}
\begin{itemize}
\item
$\{
24\alpha^2\beta^2\gamma \colon 
24\alpha\beta^2,\,
8\gamma^3 \opt 
6\alpha^4
\}, \,
\alpha=\tfrac{1}{2}\pi,\,
\beta=\tfrac{3}{4}\pi,\,
\gamma=\tfrac{2}{3}\pi$.
\item
$\{
36\alpha^2\beta^2\gamma \colon 
36\alpha\beta^2,\,
8\gamma^3 \opt 
12\alpha^3\gamma
\}, \,
\alpha=\tfrac{4}{9}\pi,\,
\beta=\tfrac{7}{9}\pi,\,
\gamma=\tfrac{2}{3}\pi$.
\item
$\{
60\alpha^2\beta^2\gamma \colon 
(48+y_2)\alpha\beta^2,\,
20\gamma^3 \opt 
(24-2y_2)\alpha^3\beta,\,
y_2\alpha^5
\}, 
\\
\alpha=\tfrac{2}{5}\pi,\,
\beta=\tfrac{4}{5}\pi,\,
\gamma=\tfrac{2}{3}\pi$.
\item
$\{
48\alpha^2\beta\gamma^2 \colon 
24\alpha\beta^2,\,
32\gamma^3 \opt 
18\alpha^4
\}, \,
\alpha=\tfrac{1}{2}\pi,\,
\beta=\tfrac{3}{4}\pi,\,
\gamma=\tfrac{2}{3}\pi$.
\item
$\{
24\alpha\beta^3\gamma \colon 
24\alpha\beta^2,\,
8\gamma^3 \opt 
6\beta^4
\}, \,
\alpha=\pi,\,
\beta=\tfrac{1}{2}\pi,\,
\gamma=\tfrac{2}{3}\pi$. 
\item
$\{
36\alpha\beta^3\gamma \colon 
36\alpha\beta^2,\,
8\gamma^3 \opt 
12\beta^3\gamma
\}, \,
\alpha=\tfrac{10}{9}\pi,\,
\beta=\tfrac{4}{9}\pi,\,
\gamma=\tfrac{2}{3}\pi$.
\item
$\{
60\alpha\beta^3\gamma \colon 
60\alpha\beta^2,\,
20\gamma^3 \opt 
12\beta^5
\}, \,
\alpha=\tfrac{6}{5}\pi,\,
\beta=\tfrac{2}{5}\pi,\,
\gamma=\tfrac{2}{3}\pi$.
\item 
$\{
f\alpha\beta\gamma^3 \colon 
x_1\alpha\beta^2,\,
x_2\gamma^3 \opt 
y_1\alpha^{\frac{f+12}{24}}\beta\gamma,\,
y_2\alpha^{\frac{f}{12}}\gamma^2,\,
y_3\alpha^{\frac{f+4}{8}}\beta,\,
y_4\alpha^{\frac{f}{6}}\gamma,\,
y_5\alpha^{\frac{f}{4}}
\}, 
\\
x_1=\tfrac{1}{2}f
-\tfrac{1}{2}y_1
-\tfrac{1}{2}y_3, 
\\
x_2=f
-\tfrac{1}{3}y_1
-\tfrac{2}{3}y_2
-\tfrac{1}{3}y_4, 
\\
y_1+2y_2+3y_3+4y_4+6y_5=12, 
\\
\alpha=\tfrac{8}{f}\pi,\,
\beta=(1-\tfrac{4}{f})\pi,\,
\gamma=\tfrac{2}{3}\pi$.
\end{itemize}
\item $\{\alpha\beta^2,\gamma^3,\alpha\delta^3\}$: \label{eg12}
\begin{itemize}
\item 
$\{
60\alpha\beta^2\gamma\delta \colon 
(60-y_1)\alpha\beta^2,\,
20\gamma^3,\,
y_1\alpha\delta^3 \opt
y_1\beta^2\delta^2,\,
(12-y_1)\delta^5
\},
\\
\alpha=\tfrac{4}{5}\pi,\,
\beta=\tfrac{3}{5}\pi,\,
\gamma=\tfrac{2}{3}\pi,\,
\delta=\tfrac{2}{5}\pi$. 
\item
$\{
84\alpha\beta^2\gamma\delta \colon 
72\alpha\beta^2,\,
20\gamma^3,\,
12\alpha\delta^3 \opt
24\beta\gamma\delta^2
\}, 
\\
\alpha=\tfrac{6}{7}\pi,\,
\beta=\tfrac{4}{7}\pi,\,
\gamma=\tfrac{2}{3}\pi,\,
\delta=\tfrac{8}{21}\pi$. 
\item
$\{
132\alpha\beta^2\gamma\delta \colon 
(120-y_1)\alpha\beta^2,\,
20\gamma^3,\,
(12+y_1)\alpha\delta^3 \opt
y_1\beta^3\delta,\,
(24-y_1)\beta\delta^4
\}, 
\\
\alpha=\tfrac{10}{11}\pi,\,
\beta=\tfrac{6}{11}\pi,\,
\gamma=\tfrac{2}{3}\pi,\,
\delta=\tfrac{4}{11}\pi$. 
\item
$\{
f\alpha\beta\gamma^2\delta \colon 
x_1\alpha\beta^2, \,
x_2\gamma^3, \,
x_3\alpha\delta^3 
\\
\opt 
y_1\alpha^{\frac{f+60}{72}}\beta\delta,\,
y_2\alpha^{\frac{f+36}{72}}\beta\gamma,\,
y_3\alpha^{\frac{f}{36}}\gamma^2,\,
y_4\alpha^{\frac{f+24}{36}}\delta^2,\,
y_5\alpha^{\frac{f+12}{36}}\gamma\delta,
\\
{}\quad y_6\alpha^{\frac{f+12}{24}}\beta,\,
y_7\alpha^{\frac{f}{18}}\gamma,\,
y_8\alpha^{\frac{f+6}{18}}\delta,\,
y_9\alpha^{\frac{f}{12}}
\}, 
\\
x_1 
=\tfrac{1}{2}f
-\tfrac{1}{2}y_1
-\tfrac{1}{2}y_2
-\tfrac{1}{2}y_6, 
\\
x_2 
=\tfrac{1}{3}f
-\tfrac{1}{3}y_2
-\tfrac{2}{3}y_3
-\tfrac{1}{3}y_5
-\tfrac{1}{3}y_7, 
\\
x_3 
=\tfrac{1}{3}f
-\tfrac{1}{3}y_1
-\tfrac{2}{3}y_4
-\tfrac{1}{3}y_5
-\tfrac{1}{3}y_8, 
\\ y_1+y_2+2y_3+2y_4+2y_5+3y_6+4y_7+4y_8+6y_9=12, 
\\
\alpha=\tfrac{24}{f}\pi,\,
\beta=(1-\tfrac{12}{f})\pi,\,
\gamma=\tfrac{2}{3}\pi,\,
\delta=(\tfrac{2}{3}-\tfrac{8}{f})\pi$.
\end{itemize}
\item $\{\alpha\beta^2,\gamma^3,\beta\delta^3\}$: \label{eg26} \label{eg33}
\begin{itemize}
\item 
$\{
f\alpha\beta^2\gamma\delta \colon 
x_1\alpha\beta^2, \,
x_2\gamma^3, \,
x_3\beta\delta^3 
\\
\opt 
y_1\alpha^{\frac{f-12}{72}}\gamma^2\delta,\,
y_2\alpha^{\frac{f+36}{72}}\beta\gamma,\,
y_3\alpha^{\frac{f-36}{72}}\gamma\delta^3,\,
y_4\alpha^{\frac{f+12}{72}}\beta\delta^2,\,
y_5\alpha^{\frac{f-60}{72}}\delta^5,
\\
{}\quad 
y_6\alpha^{\frac{f}{36}}\beta^2,\,
y_7\alpha^{\frac{f-12}{36}}\gamma\delta^2,\,
y_8\alpha^{\frac{f+12}{36}}\beta\delta,\,
y_9\alpha^{\frac{f-24}{36}}\delta^4, 
\\
{}\quad 
y_{10}\alpha^{\frac{f-4}{24}}\gamma\delta,\,
y_{11}\alpha^{\frac{f+12}{24}}\beta, \,
y_{12}\alpha^{\frac{f-12}{24}}\delta^3,
\\
{}\quad 
y_{13}\alpha^{\frac{f}{18}}\gamma,\,
y_{14}\alpha^{\frac{f-6}{18}}\delta^2,\,
y_{15}\alpha^{\frac{5f-12}{72}}\delta,\,
y_{16}\alpha^{\frac{f}{12}}
\},
\\
x_1 
=\tfrac{5}{6}f
+\tfrac{1}{6}y_1
-\tfrac{1}{2}y_2
+\tfrac{1}{6}y_3
-\tfrac{1}{6}y_4
+\tfrac{5}{6}y_5
+\tfrac{1}{3}y_7
-\tfrac{1}{3}y_8
+\tfrac{2}{3}y_9 
\\
{}\qquad 
+\tfrac{1}{6}y_{10}
-\tfrac{1}{2}y_{11}
+\tfrac{1}{2}y_{12}
+\tfrac{1}{3}y_{14}
+\tfrac{1}{6}y_{15},
\\
x_2 
=\tfrac{1}{3}f
-\tfrac{2}{3}y_1
-\tfrac{1}{3}y_2
-\tfrac{1}{3}y_3
-\tfrac{2}{3}y_6
-\tfrac{1}{3}y_7
-\tfrac{1}{3}y_{10}
-\tfrac{1}{3}y_{13},
\\
x_3 
=\tfrac{1}{3}f
-\tfrac{1}{3}y_1
-y_3
-\tfrac{2}{3}y_4
-\tfrac{5}{3}y_5
-\tfrac{2}{3}y_7
-\tfrac{1}{3}y_8
-\tfrac{4}{3}y_9 
\\
{}\qquad 
-\tfrac{1}{3}y_{10}
-y_{12}
-\tfrac{2}{3}y_{14}
-\tfrac{1}{3}y_{15}, 
\\
y_1+\cdots+y_5
+2(y_6+\cdots+y_9)
+3(y_{10}+y_{11}+y_{12}) 
\\
{}\qquad 
+4(y_{13}+y_{14})
+5y_{15}
+6y_{16}=12, 
\\
\alpha=\tfrac{24}{f}\pi,\,
\beta=(1-\tfrac{12}{f})\pi,\,
\gamma=\tfrac{2}{3}\pi,\,
\delta=(\tfrac{1}{3}+\tfrac{4}{f})\pi$.
\item
$\{
84\alpha\beta\gamma^2\delta \colon 
36\alpha\beta^2,\,
56\gamma^3,\,
12\beta\delta^3 \opt
24\alpha^2\delta^2
\}, 
\\
\alpha=\tfrac{4}{7}\pi,\,
\beta=\tfrac{9}{14}\pi,\,
\gamma=\tfrac{2}{3}\pi,\,
\delta=\tfrac{1}{7}\pi$.
\item
$\{
228\alpha\beta\gamma^2\delta \colon 
84\alpha\beta^2,\,
152\gamma^3,\,
60\beta\delta^3 \opt
48\alpha^3\delta
\}, 
\\
\alpha=\tfrac{10}{19}\pi,\,
\beta=\tfrac{14}{19}\pi,\,
\gamma=\tfrac{2}{3}\pi,\,
\delta=\tfrac{8}{19}\pi$.
\end{itemize}
\item $\{\alpha\beta^2,\gamma^3,\gamma\delta^3\}$:
\begin{itemize}
\item 
$\{
36\alpha\beta^2\gamma\delta \colon 
36\alpha\beta^2,\,
8\gamma^3,\,
12\gamma\delta^3 \}, \,
\alpha+2\beta=2\pi,\,
\gamma=\tfrac{2}{3}\pi,\,
\delta=\tfrac{4}{9}\pi$.
\item
$\{
72\alpha\beta\gamma^2\delta \colon 
36\alpha\beta^2,\,
44\gamma^3,\,
12\gamma\delta^3 \opt
18\alpha^2\delta^2
\}, 
\\
\alpha=\tfrac{5}{9}\pi,\,
\beta=\tfrac{13}{18}\pi,\,
\gamma=\tfrac{2}{3}\pi,\,
\delta=\tfrac{4}{9}\pi$.
\item
$\{
108\alpha\beta\gamma^2\delta \colon 
54\alpha\beta^2,\,
62\gamma^3,\,
30\gamma\delta^3 \opt
18\alpha^3\delta
\}, 
\\
\alpha=\tfrac{14}{27}\pi,\,
\beta=\tfrac{20}{27}\pi,\,
\gamma=\tfrac{2}{3}\pi,\,
\delta=\tfrac{4}{9}\pi$.
\item
$\{
144\alpha\beta\gamma^2\delta \colon 
72\alpha\beta^2,\,
80\gamma^3,\,
48\gamma\delta^3 \opt
18\alpha^4
\}, 
\\
\alpha=\tfrac{1}{2}\pi,\,
\beta=\tfrac{3}{4}\pi,\,
\gamma=\tfrac{2}{3}\pi,\,
\delta=\tfrac{4}{9}\pi$.
\end{itemize}
\item $\{\alpha\beta^2,\gamma^3,\delta^4\}$: \label{eg32}
\begin{itemize}
\item 
$\{
24\alpha\beta^2\gamma\delta \colon 
24\alpha\beta^2,\,
8\gamma^3,\,
6\delta^4 \}, \,
\alpha+2\beta=2\pi,\,
\gamma=\tfrac{2}{3}\pi,\,
\delta=\tfrac{1}{2}\pi$.
\item
$\{
72\alpha\beta\gamma^2\delta \colon 
36\alpha\beta^2,\,
44\gamma^3,\,
18\delta^4 \opt
12\alpha^3\gamma
\}, 
\\
\alpha=\tfrac{4}{9}\pi,\,
\beta=\tfrac{7}{9}\pi,\,
\gamma=\tfrac{2}{3}\pi,\,
\delta=\tfrac{1}{2}\pi$.
\item
$\{
96\alpha\beta\gamma^2\delta \colon 
48\alpha\beta^2,\,
56\gamma^3,\,
18\delta^4 \opt
24\alpha^2\gamma\delta
\}, 
\\
\alpha=\tfrac{5}{12}\pi,\,
\beta=\tfrac{19}{24}\pi,\,
\gamma=\tfrac{2}{3}\pi,\,
\delta=\tfrac{1}{2}\pi$.
\item
$\{
120\alpha\beta\gamma^2\delta \colon 
(48+y_2)\alpha\beta^2,\,
80\gamma^3,\,
30\delta^4 \opt
(24-2y_2)\alpha^3\beta,\,
y_2\alpha^5
\}, 
\\
\alpha=\tfrac{2}{5}\pi,\,
\beta=\tfrac{4}{5}\pi,\,
\gamma=\tfrac{2}{3}\pi,\,
\delta=\tfrac{1}{2}\pi$.
\item
$\{
192\alpha\beta\gamma^2\delta \colon 
96\alpha\beta^2,\,
128\gamma^3,\,
42\delta^4 \opt
24\alpha^4\delta
\}, 
\\
\alpha=\tfrac{3}{8}\pi,\,
\beta=\tfrac{13}{16}\pi,\,
\gamma=\tfrac{2}{3}\pi,\,
\delta=\tfrac{1}{2}\pi$.
\end{itemize}
\item $\{\alpha\beta^2,\gamma^3,\delta^5\}$:  \label{eg34}
\begin{itemize}
\item 
$\{
60\alpha\beta^2\gamma\delta \colon 
(60-y_1)\alpha\beta^2,\,
20\gamma^3,\,
(12-y_1)\delta^5 \opt
y_1\alpha\delta^3,\,
y_1\beta^2\delta^2
\},
\\
\alpha+2\beta=2\pi,\,
\gamma=\tfrac{2}{3}\pi,\,
\delta=\tfrac{2}{5}\pi$. 
\item
$\{
120\alpha\beta\gamma^2\delta \colon 
60\alpha\beta^2,\,
80\gamma^3,\,
12\delta^5 \opt
30\alpha^2\delta^2
\}, 
\\
\alpha=\tfrac{3}{5}\pi,\,
\beta=\tfrac{7}{10}\pi,\,
\gamma=\tfrac{2}{3}\pi,\,
\delta=\tfrac{2}{5}\pi$.
\end{itemize}
\end{itemize}

\subsection{Four Distinct Angles at Degree $3$ Vertices}
\label{classify4}

Suppose there are four distinct angles at degree $3$ vertices. Then the pentagon is $\alpha^2\beta\gamma\delta,\alpha\beta^2\gamma\delta,\alpha\beta\gamma^2\delta,\alpha\beta\gamma\delta^2,\alpha\beta\gamma\delta\epsilon$. The possible degree $3$ vertices are listed in the four angle part of Table \ref{deg3AVC}. 

For $\{\alpha\beta\gamma,\alpha\delta^2\}$, we get the following possible angle combinations:
\begin{enumerate}
\item If $\alpha\beta\gamma,\alpha\delta^2$ are all the degree $3$ vertices, then by Lemma \ref{more3}, $\alpha$ appears twice in the pentagon. Therefore the pentagon is $\alpha^2\beta\gamma\delta$, and we get $\{\alpha^2\beta\gamma\delta\colon \alpha\beta\gamma,\alpha\delta^2\}$.
\item If the tiling has four distinct angles, then the pentagon has angle combination $\alpha^2\beta\gamma\delta,\alpha\beta^2\gamma\delta,\alpha\beta\gamma^2\delta,\alpha\beta\gamma\delta^2$. We need to consider one of these four combined with one of $\{\alpha\beta\gamma,\alpha\delta^2,\beta^2\delta\},\{\alpha\beta\gamma,\alpha\delta^2,\beta^3\}$.
\item If the tiling has four distinct angles, then the pentagon has angle combination $\alpha\beta\gamma\delta\epsilon$. Moreover, by applying Lemma \ref{hdeg} to $\epsilon$, we know one of $\alpha\epsilon^3,\beta\epsilon^3,\gamma\epsilon^3,\delta\epsilon^3,\epsilon^4,\epsilon^5$ is a vertex. We need to consider one of $\{\alpha\beta\gamma\delta\epsilon\colon \alpha\beta\gamma,\alpha\delta^2,\beta^2\delta\},\{\alpha\beta\gamma\delta\epsilon\colon \alpha\beta\gamma,\alpha\delta^2,\beta^3\}$ combined with one of the six vertices with $\epsilon$.
\end{enumerate}
The first AVC $\{\alpha^2\beta\gamma\delta\colon \alpha\beta\gamma,\alpha\delta^2\}$ is Example 3 in Section \ref{routine3}. The order matrix is invertible in all the other combinations, and the routines in Sections \ref{routine1} and \ref{routine2} can be applied. 

The case $\{\alpha\beta\gamma,\alpha^2\delta\}$ is completely similar to $\{\alpha\beta\gamma,\alpha\delta^2\}$. One combination can be calculated similar to Example 3 in Section \ref{routine3}. The remaining combinations can be calculated by the routines in Sections \ref{routine1} and \ref{routine2}.

For $\{\alpha\beta\gamma,\delta^3\}$, we note that the angle sums for $\{\alpha\beta\gamma\delta^2\colon \alpha\beta\gamma,\delta^3\}$ imply $f=12$, and the case is dismissed. Then up to symmetry, we get the following possible combinations:
\begin{enumerate}
\item $\{\alpha^2\beta\gamma\delta\colon \alpha\beta\gamma,\delta^3\}$.
\item $\{\alpha\beta\gamma\delta\epsilon\colon \alpha\beta\gamma,\delta^3,\alpha\epsilon^3\}$.
\item $\{\alpha\beta\gamma\delta\epsilon\colon \alpha\beta\gamma,\delta^3\}$, together with one of $\delta\epsilon^3,\epsilon^4,\epsilon^5$.
\end{enumerate}
The first has positive lower bounds for $\alpha,\delta$, and the second has positive lower bounds for $\alpha,\delta,\epsilon$. Both can be calculated similar to Example 1 in Section \ref{routine3}. The third is Example 4 in Section \ref{routine3}.

For $\{\alpha\beta^2,\gamma\delta^2\}$, if the optional vertex $\alpha^2\delta$ also appears, then the routines in Sections \ref{routine1} and \ref{routine2} can be applied. If $\alpha\beta^2,\gamma\delta^2$ are the only degree $3$ vertices, then by Lemma \ref{more3} (and the remark after the lemma), we know $\beta,\delta$ together appear at least three times in the pentagon. Then up to the symmetry of exchanging $\beta,\delta$, we may assume the pentagon is $\alpha\beta^2\gamma\delta$. The angle sums of $\alpha\beta^2,\gamma\delta^2$ and the angle sum for $\alpha\beta^2\gamma\delta$ imply
\[
\alpha+2\beta=2\pi,\,
\gamma=\tfrac{8\pi}{f},\,
\delta=(1-\tfrac{4}{f})\pi.
\]
Then we may calculate similar to Example 3 in Section \ref{routine3}.

For $\{\alpha\beta^2,\alpha^2\gamma,\delta^3\}$, the routines in Sections \ref{routine1} and \ref{routine2} can be applied to all the cases.

We are able to find all the full AVCs for angle congruent pentagonal tilings with four distinct angles at degree $3$ vertices. The whole list is too long to be included in this paper.

\subsection{Five Distinct Angles at Degree $3$ Vertices}
\label{classify5}

Suppose there are five distinct angles at degree $3$ vertices. Then the pentagon is $\alpha\beta\gamma\delta\epsilon$. The possible degree $3$ vertices are listed in the five angle part of Table \ref{deg3AVC}.

The case $\{\alpha\beta\gamma\delta\epsilon\colon \alpha\beta\gamma,\alpha\delta^2,\alpha^2\epsilon\}$ is the simplest. By applying Lemma \ref{more3} to $\alpha$, we know one of the optional degree $3$ vertices $\beta\epsilon^2,\beta^2\delta,\beta^3$ must appear. Then the routines in Sections \ref{routine1} and \ref{routine2} can be applied. 

For $\{\alpha\beta\gamma\delta\epsilon\colon\alpha\beta\gamma,\alpha\delta\epsilon\}$, by applying Lemma \ref{more3} to $\alpha$, at least one of the optional vertices must appear. If $\beta\delta^2$ is a vertex, then
\[
\alpha=(1-\tfrac{4}{f})\pi,\,
\gamma=(1+\tfrac{4}{f})\pi-\beta,\,
\delta=\pi-\tfrac{1}{2}\beta,\,
\epsilon=\tfrac{4}{f}\pi+\tfrac{1}{2}\beta.
\]
The situation does not appear in Sections \ref{routine1}, \ref{routine2}, \ref{routine3}. The difficulty also appears for other combinations. For example, for $\{\alpha\beta\gamma\delta\epsilon\colon\alpha\beta\gamma,\alpha\delta^2,\beta\epsilon^2\}$, we have
\[
\beta=(2-\tfrac{8}{f})\pi-\alpha,\,
\gamma=\tfrac{8}{f}\pi,\,
\delta=\pi-\tfrac{1}{2}\alpha,\,
\epsilon=\tfrac{4}{f}\pi+\tfrac{1}{2}\alpha.
\]
For $\{\alpha\beta\gamma\delta\epsilon\colon\alpha\beta\gamma,\alpha\delta^2,\delta\epsilon^2\}$, we have
\[
\alpha=(2-\tfrac{16}{f})\pi,\,
\beta+\gamma=\tfrac{16}{f}\pi,\,
\delta=\tfrac{8}{f}\pi,\,
\epsilon=(1-\tfrac{4}{f})\pi.
\]
Although some of the techniques we developed so far may still be used to compute some of the full AVCs, we choose not to pursue further. 

We remark that the examples above allow free continuous choice of at most one angle. Among the AVCs with five distinct angles at degree $3$ vertices, the only possible one allowing free continuous choices of two angles is $\{\alpha\beta\gamma\delta\epsilon\colon\alpha\beta\gamma,\delta\epsilon^2\}$. The angle sums of $\alpha\beta\gamma,\delta\epsilon^2$ and the angle sum for $\alpha\beta\gamma\delta\epsilon$ imply
\[
\alpha+\beta+\gamma=2\pi,\,
\delta=\tfrac{8}{f}\pi,\,
\epsilon=(1-\tfrac{4}{f})\pi.
\]
The appearance of one more vertex will cut the number of free choices. For example, according to Table \ref{deg3AVC}, we should consider the case $\alpha^3$ is also a vertex. In this case, we get 
\[
\alpha=\tfrac{2}{3}\pi,\,
\beta+\gamma=\tfrac{4}{3}\pi,\,
\delta=\tfrac{8}{f}\pi,\,
\epsilon=(1-\tfrac{4}{f})\pi,
\]
which allows free continuous choices of only one angle.

Now assume $\alpha^3$ is not a vertex, which means $\alpha\beta\gamma,\delta\epsilon^2$ are the only degree $3$ vertices. Up to the symmetry of exchanging $\alpha,\beta,\gamma$, we may assume $\alpha<\beta<\gamma$. By $\alpha+\beta+\gamma=2\pi$, this implies $\beta<\pi$.

By $f>12$, there must be a high degree vertex. By applying Lemma \ref{more3} to $\alpha,\beta,\gamma$, and the fact that the pentagon has only five angles, we know one of $\alpha,\beta,\gamma$ appears at high degree vertex. Then by applying Lemma \ref{counting} to any two from $\alpha,\beta,\gamma$, we know each of $\alpha,\beta,\gamma$ appears at high degree vertices. Moreover, there is a high degree vertex $V=\gamma\cdots$, such that the number of $\alpha$ is no more than the number of $\gamma$. By $\alpha+\beta+\gamma=2\pi$ and $\alpha<\beta<\gamma$, if $V$ has $\alpha$, then $V=\alpha\gamma\cdots$, with no $\alpha,\beta,\gamma$ in $\cdots$. By $\delta+\epsilon>\pi>\beta$, and $\delta<\epsilon$, and $\alpha+\beta+\gamma=2\pi$, we know $V=\alpha\gamma\delta^g$. If $V$ has no $\alpha$, then by an argument similar to Example 4 in Section \ref{routine3}, we get $V=\beta\gamma\delta^g,\gamma\delta^g,\gamma\delta^g\epsilon,\gamma^2\delta^g$. Then the angle sum of $V$ further implies
\begin{align*}
\alpha\gamma\delta^g &\colon
\alpha+\gamma=(2-\tfrac{8g}{f})\pi,\,
\beta=\tfrac{8g}{f}\pi,\,
\delta=\tfrac{8}{f}\pi,\,
\epsilon=(1-\tfrac{4}{f})\pi; \\
\beta\gamma\delta^g
&\colon
\alpha=\tfrac{8g}{f}\pi,\,
\beta+\gamma=(2-\tfrac{8g}{f})\pi; \\
\gamma\delta^g
&\colon
\alpha+\beta=\tfrac{8g}{f}\pi,\,
\gamma=(2-\tfrac{8g}{f})\pi; \\
\gamma\delta^g\epsilon
&\colon
\alpha+\beta=(1+\tfrac{8g-4}{f})\pi,\,
\gamma=(1-\tfrac{8g-4}{f})\pi; \\
\gamma^2\delta^g
&\colon
\alpha+\beta=(1+\tfrac{4g}{f})\pi,\,
\gamma=(1-\tfrac{4g}{f})\pi.
\end{align*}
Each case allows free continuous choices of only one angle.

We conclude that, if there are five distinct angles at degree $3$ vertices, then the AVC does not allow free continuous choice of two angles.

\section{Tilings with Two Free Angles}
\label{realize}

By free, we mean continuous choice. By two free angles, we mean two angles can be any point in a two dimensional domain.

In the complete classification of edge-to-edge tilings of the sphere by (angle as well as edge) congruent pentagons \cite{ay,awy,cly5,gsy,wy1,wy2}, we find the tilings are three pentagonal subdivisions, two double pentagonal subdivisions, two earth map tilings, and the modifications of these tilings. The pentagonal subdivisions allow two free parameters, which can be two angles. In fact, the two dimensional moduli of pentagonal subdivisions are described in detail in \cite{wy3}. The double pentagonal subdivisions are rigid, which means the pentagons have specific angle values and edge lengths. For each fixed $f=0$ mod $4$, one earth map tiling is rigid, and one earth map tiling allows one free parameter. The modifications often require extra equalities among angles, and therefore cuts the number of free parameters. In the end, the pentagonal subdivisions are the only tilings allowing two free parameters.

How many free angles can angle congruent tilings have? Gao, Shi, Yan \cite[Proposition 17]{gsy} listed all the full AVCs for angle congruent dodecahedron tilings. The dodecahedron is the pentagonal subdivision of the tetrahedron, and is also  the pentagonal tiling with smallest number $f=12$ of tiles. In the list, the only AVC allowing two free angles is $\{12\alpha^2\beta\gamma\delta\colon 8\alpha^3,12\beta\gamma\delta\}$. To make the notations consistent with the other AVCs in this paper, we exchange $\alpha,\delta$, and the AVC becomes $\{12\alpha\beta\gamma\delta^2\colon 12\alpha\beta\gamma,8\delta^3\}$. In \cite[Proposition 19]{gsy}, Gao, Shi, Yan further showed angle congruent tilings for this AVC are given by the first of Figure \ref{12angletiling}, in which $\beta,\gamma$ can be independently exchanged along six dashed edges. 

\begin{figure}[htp]
\centering
\begin{tikzpicture}[>=latex]

\foreach \b in {0,1,2,3,4}
{

\coordinate  (A\b) at (90+72*\b:0.7);
\coordinate  (B\b) at (90+72*\b:1.3);
\coordinate  (C\b) at (126+72*\b:1.7);
\coordinate  (D\b) at (54+72*\b:2.5);

\coordinate  (P\b) at (90+72*\b:0.45);

\coordinate  (Q\b) at (102+72*\b:0.83);
\coordinate  (R\b) at (100+72*\b:1.13);
\coordinate  (S\b) at (126+72*\b:1.43);
\coordinate  (T\b) at (152+72*\b:1.13);
\coordinate  (U\b) at (150+72*\b:0.83);

\coordinate  (X\b) at (90+72*\b:1.5);
\coordinate  (Y\b) at (60+72*\b:1.8);
\coordinate  (Z\b) at (120+72*\b:1.8);
\coordinate  (V\b) at (59+72*\b:2.15);
\coordinate  (W\b) at (121+72*\b:2.15);

\coordinate  (L\b) at (54+72*\b:2.7);
}

\draw
	(A3) -- (A4) -- (A0) -- (A1) -- (A2)
	(A1) -- (B1) -- (C0) -- (D1) -- (D2)
	(A2) -- (B2) -- (C1) -- (D2) -- (D3)
	(A3) -- (B3) -- (C3) -- (D4) -- (D3)
	(A4) -- (B4) -- (C4) -- (D0) -- (D4)
	(C4) -- (B0) -- (C0)
	(B2) -- (C2) -- (B3);
	
\draw[dashed]
	(A0) -- (B0)
	(A2) -- (A3)
	(B1) -- (C1)
	(B4) -- (C3)
	(D0) -- (D1)
	(C2) -- (D3);

\node at (P0) {\small $\alpha$};
\node at (P1) {\small $\delta$};
\node at (P2) {\small $\beta$};
\node at (P3) {\small $\gamma$};
\node at (P4) {\small $\delta$};

\node at (Q0) {\small $\beta$};
\node at (R0) {\small $\gamma$};
\node at (S0) {\small $\delta$};
\node at (T0) {\small $\alpha$};
\node at (U0) {\small $\delta$};

\node at (Q1) {\small $\delta$};
\node at (R1) {\small $\beta$};
\node at (S1) {\small $\gamma$};
\node at (T1) {\small $\delta$};
\node at (U1) {\small $\alpha$};

\node at (Q2) {\small $\gamma$};
\node at (R2) {\small $\delta$};
\node at (S2) {\small $\alpha$};
\node at (T2) {\small $\delta$};
\node at (U2) {\small $\beta$};

\node at (Q3) {\small $\alpha$};
\node at (R3) {\small $\delta$};
\node at (S3) {\small $\beta$};
\node at (T3) {\small $\gamma$};
\node at (U3) {\small $\delta$};

\node at (Q4) {\small $\delta$};
\node at (R4) {\small $\alpha$};
\node at (S4) {\small $\delta$};
\node at (T4) {\small $\beta$};
\node at (U4) {\small $\gamma$};

\node at (X0) {\small $\alpha$};
\node at (Y0) {\small $\delta$};
\node at (Z0) {\small $\delta$};
\node at (V0) {\small $\beta$};
\node at (W0) {\small $\gamma$};

\node at (X1) {\small $\gamma$};
\node at (Y1) {\small $\delta$};
\node at (Z1) {\small $\beta$};
\node at (V1) {\small $\alpha$};
\node at (W1) {\small $\delta$};

\node at (X2) {\small $\delta$};
\node at (Y2) {\small $\alpha$};
\node at (Z2) {\small $\gamma$};
\node at (V2) {\small $\delta$};
\node at (W2) {\small $\beta$};

\node at (X3) {\small $\delta$};
\node at (Y3) {\small $\beta$};
\node at (Z3) {\small $\alpha$};
\node at (V3) {\small $\gamma$};
\node at (W3) {\small $\delta$};

\node at (X4) {\small $\beta$};
\node at (Y4) {\small $\gamma$};
\node at (Z4) {\small $\delta$};
\node at (V4) {\small $\delta$};
\node at (W4) {\small $\alpha$};

\node at (L0) {\small $\gamma$};
\node at (L1) {\small $\beta$};
\node at (L2) {\small $\delta$};
\node at (L3) {\small $\alpha$};
\node at (L4) {\small $\delta$};


\foreach \a in {0,...,3}
{
\begin{scope}[xshift=7cm, rotate=90*\a]

\draw
	(0,0) -- (0.8,0) -- (1.2,0.6) -- (0.6,1.2) -- (0,0.8)
	(1.2,0.6) -- (1.8,0.6)
	(1.2,-0.6) -- (2.6,-0.6)
	(1.8,-2.6) -- (1.8,1.8)
	(2.6,-2.6) -- (2.6,2.6) -- (3,3);

\node at (0.6,0.95) {\small $\beta$};
\node at (0.95,0.6) {\small $\delta$}; 
\node at (0.2,0.7) {\small $\gamma$};
\node at (0.7,0.2) {\small $\alpha$};
\node at (0.2,0.2) {\small $\epsilon$};

\node at (0.4,1.3) {\small $\gamma$}; 
\node at (0,1.1) {\small $\beta$};  
\node at (0.4,1.6) {\small $\epsilon$};
\node at (-0.4,1.3) {\small $\delta$};
\node at (-0.4,1.6) {\small $\alpha$};

\node at (0.8,1.6) {\small $\epsilon$}; 
\node at (1.6,1.6) {\small $\gamma$};
\node at (0.8,1.3) {\small $\alpha$};
\node at (1.3,0.8) {\small $\delta$};
\node at (1.6,0.8) {\small $\beta$};

\node at (2.4,-0.4) {\small $\alpha$};
\node at (2,1.6) {\small $\beta$};
\node at (2,0.6) {\small $\gamma$};
\node at (2.4,1.6) {\small $\delta$};
\node at (2,-0.4) {\small $\epsilon$}; 

\node at (2,-1.8) {\small $\alpha$}; 
\node at (2.4,-2.4) {\small $\beta$}; 
\node at (2.4,-0.8) {\small $\gamma$}; 
\node at (2,-2.4) {\small $\delta$}; 
\node at (2,-0.8) {\small $\epsilon$}; 

\node at (2.8,2.5) {\small $\alpha$};
\node at (2.8,-0.6) {\small $\beta$};
\node at (2.8,-2.5) {\small $\gamma$};
\node at (2.8,1.8) {\small $\delta$};
\node at (3.1,0) {\small $\epsilon$};

\end{scope}
}

\end{tikzpicture}
\caption{Angle congruent tilings for $f=12,24$.}
\label{12angletiling}
\end{figure}

For $f>12$, by the discussion in Section \ref{classify}, the only AVCs allowing two free angles are \eqref{60a}, \eqref{36a}, \eqref{24a} from Example 4 of Section \ref{routine3}, with respective $f=60,36,24$. 

In fact, for the purpose of finding tilings allowing two free angles, we do not need to calculate the full AVC. What we need is to compare the number of angles with the number of equations satisfied by the angles. For up to three distinct angles at degree $3$ vertices, Table \ref{3completeAVC} gives all the possibilities. For example, for the pentagon with edge combination $\alpha^2\beta\gamma\delta$, we know either $\alpha\beta\gamma,\alpha^3$ are vertices, or $\alpha\beta^2,\alpha^2\gamma$ are vertices. Moreover, by Lemma \ref{hdeg}, we also know one of $\alpha\delta^3,\beta\delta^3,\gamma\delta^3,\delta^4,\delta^5$ is a vertex. Therefore we have four angles satisfying $1+2+1=4$ equations. In principle (upon checking the dependency of equations), this means no free angle. 

For four distinct angles at degree $3$ vertices, by Table \ref{deg3AVC}, either the pentagon has four angles satisfying at least $1+2=3$ equations, or the pentagon has five angles satisfying at least $1+2+1=4$ equations. In principle, this means at most one free angle. Upon checking the dependency of equations, we find \eqref{60a}, \eqref{36a}, \eqref{24a} are the only cases with two free angles.

For five distinct angles at degree $3$ vertices, we go through the same discussion using Table \ref{deg3AVC}. The only case that requires further discussion (i.e., possibly 2 free angle) is $\{\alpha\beta\gamma\delta\epsilon\colon\alpha\beta\gamma,\delta\epsilon^2\}$. The case is discussed in detail Section \ref{classify5}, and we find one more equation from high order vertex that cuts down the number of free angle to no more than 1.

We already know that the AVCs \eqref{60a} and \eqref{24a} can be realized by pentagonal subdivisions. It turns out that the AVC \eqref{36a} cannot be realized.

\begin{theorem}\label{anglerealize}
Edge-to-edge tilings of the sphere by angle congruent pentagons that allow two free angles are the pentagonal subdivision tilings, and the further modification by independent exchange of $\beta,\gamma$ along six dashed edges in the first of Figure \ref{12angletiling}.
\end{theorem}

The pentagonal subdivision of the octahedron (and the cube) has $f=24$ tiles, and the angle congruent tiling is the second of Figure \ref{12angletiling}. The pentagonal subdivision of the icosahedron (and the dodecahedron) has $f=60$ tiles, and the angle congruent tiling is similar.

\begin{proof}
We need to find angle congruent tilings for the AVCs \eqref{60a}, \eqref{36a}, \eqref{24a}. 

First, we argue that $\delta,\epsilon$ are not adjacent in the pentagon. In the three AVCs, we have $\delta\cdots=\delta^3,\delta\epsilon^3$ and $\epsilon\cdots=\delta\epsilon^3,\epsilon^4,\epsilon^5$. In particular, we have consecutive $\epsilon\epsilon\epsilon$ at the vertex $\epsilon\cdots$. In the first of Figure \ref{prooffig}, we let the tile $T_1$ contain the middle $\epsilon$. Then we may assume $\delta_1$ is located as indicated. By $\delta_1\cdots=\delta^3,\delta\epsilon^3$, the angle in $T_2$ next to $\delta_1$ is $\delta,\epsilon$. Since we already have $\epsilon_2$, we know the angle is $\delta_2$, as indicated in the picture. Then $\delta_1\delta_2\cdots=\delta^3$. This gives $\delta_3$ in tile $T_3$. Then $\epsilon_3$ is at one of the two places indicated in the picture. Then by $\epsilon\cdots=\delta\epsilon^3,\epsilon^4,\epsilon^5$, we know the angle marked $?$ is $\delta$ or $\epsilon$. This implies $T_1$ or $T_2$ has two $\delta$ or two $\epsilon$, a contradiction.

By $\delta,\epsilon$ not adjacent, and by the symmetry of exchanging $\alpha,\beta,\gamma$ in all three AVCs, we may assume the angles in the pentagon are arranged in the order $\alpha,\delta,\beta,\gamma,\epsilon$, like the tiles in the second and third of Figure \ref{prooffig}.

\begin{figure}[htp]
\centering
\begin{tikzpicture}[>=latex,scale=1]

\begin{scope}[xshift=-4cm]

\draw
	(-0.4,0) -- (0.8,0) -- (1.2,0.6) -- (0.6,1.2) -- (0,0.8) -- (0,-0.8) -- (0.6,-1.2) -- (1.2,-0.6) -- (0.8,0)
	(1.2,0.6) -- (1.8,0.6) -- (1.8,-0.6) -- (1.2,-0.6);

\node at (-0.2,-0.2) {\small $\epsilon$};

\node at (0.95,0.55) {\small ?};
\node at (0.95,0.55) {\small ?};

\node at (0.2,0.2) {\small $\epsilon$};
\node at (0.7,0.2) {\small $\delta$};

\node at (0.2,-0.2) {\small $\epsilon$};
\node at (0.7,-0.2) {\small $\delta$};

\node at (1,0) {\small $\delta$};
\node at (1.3,0.4) {\small $\epsilon$};
\node at (1.3,-0.4) {\small $\epsilon$};

\node[inner sep=0.5,draw,shape=circle] at (0.5,-0.5) {\small 1};
\node[inner sep=0.5,draw,shape=circle] at (0.5,0.5) {\small 2};
\node[inner sep=0.5,draw,shape=circle] at (1.45,0) {\small 3};

\end{scope}


\foreach \x in {0,1,2,3}
{
\begin{scope}[rotate=90*\x]

\draw
	(0,0) -- (0.8,0) -- (1.2,0.6) -- (0.6,1.2) -- (0,0.8);

\node at (0.2,0.2) {\small $\epsilon$};
\node at (0.7,0.2) {\small $\gamma$};
\node at (0.95,0.57) {\small $\beta$};
\node at (0.6,0.95) {\small $\delta$};
\node at (0.2,0.7) {\small $\alpha$};

\end{scope}
}

\draw
	(0.6,1.2) -- (0.6,1.8) -- (1.8,1.8) -- (1.8,-0.6) -- (1.2,-0.6)
	(1.2,0.6) -- (2.3,0.6);

\node at (1.05,0.05) {\small $\beta$};
\node at (1.3,-0.4) {\small $\delta$};
\node at (1.3,0.4) {\small $\gamma$};
\node at (1.65,-0.4) {\small $\alpha$};
\node at (1.65,0.4) {\small $\epsilon$};
\node at (1.25,0.8) {\small $\alpha$};
\node at (1.65,0.8) {\small $\epsilon$};
\node at (1.65,1.6) {\small $\gamma$};
\node at (0.8,1.55) {\small $\beta$};
\node at (0.8,1.25) {\small $\delta$};
\node at (1.95,0.4) {\small $\epsilon$};
\node at (1.95,0.8) {\small $\epsilon$};

\node at (1.25,-0.8) {\small $\delta$};
\node at (0.4,1.25) {\small $\delta$};

\node[inner sep=0.5,draw,shape=circle] at (0.5,-0.5) {\small 1};
\node[inner sep=0.5,draw,shape=circle] at (0.5,0.5) {\small 2};
\node[inner sep=0.5,draw,shape=circle] at (-0.5,0.5) {\small 3};
\node[inner sep=0.5,draw,shape=circle] at (-0.5,-0.5) {\small 4};
\node[inner sep=0.5,draw,shape=circle] at (1.45,0) {\small 5};
\node[inner sep=0.5,draw,shape=circle] at (1.3,1.3) {\small 6};


\begin{scope}[xshift=5cm]

\foreach \x in {0,1,2,3}
\draw[rotate=90*\x]
	(0,0) -- (0.8,0) -- (1.2,0.6) -- (0.6,1.2) -- (0,0.8);

\draw
	(-1.2,-0.6) -- (-1.8,-0.6) -- (-1.8,-2.5) -- (2.5,-2.5) -- (2.5,0.6)
	(0.6,1.2) -- (0.6,1.8) -- (3.2,1.8) -- (3.2,-1.8) -- (-1.8,-1.8)
	(-0.6,-1.2) -- (-0.6,-1.8)
	(0.6,-1.2) -- (0.6,-2.5)
	(1.8,1.8) -- (1.8,-1.8)
	(1.2,-0.6) -- (1.8,-0.6)
	(1.2,0.6) -- (3.2,0.6);
	
\node at (0.2,0.2) {\small $\epsilon$};
\node at (0.7,0.2) {\small $\gamma$};
\node at (0.95,0.55) {\small $\beta$};
\node at (0.55,0.95) {\small $\delta$};
\node at (0.2,0.7) {\small $\alpha$};
\node at (0.2,-0.2) {\small $\epsilon$};
\node at (0.7,-0.2) {\small $\alpha$};
\node at (0.95,-0.6) {\small $\delta$};
\node at (0.6,-0.95) {\small $\beta$};
\node at (0.2,-0.7) {\small $\gamma$};
\node at (-0.2,0.2) {\small $\delta$};
\node at (-0.7,0.2) {\small $\alpha$};
\node at (-0.95,0.6) {\small $\epsilon$};
\node at (-0.6,0.95) {\small $\gamma$};
\node at (-0.2,0.65) {\small $\beta$};
\node at (-0.2,-0.2) {\small $\epsilon$};
\node at (-0.7,-0.2) {\small $\gamma$};
\node at (-0.95,-0.55) {\small $\beta$};
\node at (-0.55,-0.95) {\small $\delta$};
\node at (-0.2,-0.7) {\small $\alpha$};
\node at (1.05,0) {\small $\beta$};
\node at (1.3,-0.4) {\small $\delta$};
\node at (1.3,0.4) {\small $\gamma$};
\node at (1.6,-0.4) {\small $\alpha$};
\node at (1.6,0.4) {\small $\epsilon$};
\node at (1.3,-0.8) {\small $\delta$};
\node at (1.6,-0.85) {\small $\beta$};
\node at (1.6,-1.6) {\small $\gamma$};
\node at (0.8,-1.6) {\small $\epsilon$};
\node at (0.8,-1.25) {\small $\alpha$};
\node at (0,-1.05) {\small $\beta$};
\node at (0.4,-1.25) {\small $\gamma$};
\node at (-0.4,-1.25) {\small $\delta$};
\node at (0.4,-1.6) {\small $\epsilon$};
\node at (-0.4,-1.6) {\small $\alpha$};
\node at (-0.8,-1.6) {\small $\beta$};
\node at (-0.8,-1.25) {\small $\delta$};
\node at (-1.3,-0.8) {\small $\alpha$};
\node at (-1.6,-0.8) {\small $\epsilon$};
\node at (-1.6,-1.6) {\small $\gamma$};
\node at (-0.6,-2) {\small $\gamma$};
\node at (0.4,-2) {\small $\epsilon$};
\node at (-1.6,-2) {\small $\beta$};
\node at (-1.6,-2.3) {\small $\delta$};
\node at (0.4,-2.3) {\small $\alpha$};
\node at (0.8,-2) {\small $\delta$};
\node at (0.8,-2.3) {\small $\beta$};
\node at (1.8,-2) {\small $\alpha$};
\node at (2.3,-2.3) {\small $\gamma$};
\node at (2.3,-2) {\small $\epsilon$};
\node at (2,0.4) {\small $\epsilon$};
\node at (2,-1.6) {\small $\beta$};
\node at (2.3,-1.55) {\small $\delta$};
\node at (2,-0.6) {\small $\gamma$};
\node at (2.3,0.4) {\small $\alpha$};
\node at (3,-1.6) {\small $\alpha$};
\node at (3,0.4) {\small $\beta$};
\node at (2.7,0.4) {\small $\gamma$};
\node at (3,-0.6) {\small $\delta$};
\node at (2.7,-1.6) {\small $\epsilon$};
\node at (3,0.8) {\small $\gamma$};
\node at (2,0.8) {\small $\delta$};
\node at (2.5,0.8) {\small $\beta$};
\node at (2,1.6) {\small $\alpha$};
\node at (3,1.6) {\small $\epsilon$};
\node at (1.3,0.8) {\small $\alpha$};
\node at (1.6,0.8) {\small $\epsilon$};
\node at (1.6,1.6) {\small $\gamma$};
\node at (0.8,1.6) {\small $\beta$};
\node at (0.8,1.3) {\small $\delta$};
\node at (0.4,1.3) {\small $\delta$};
\node at (0,1.05) {\small $\gamma$};
\node at (2.7,-2) {\small $\epsilon$};

\node[inner sep=0.5,draw,shape=circle] at (0.5,0.5) {\small $2$};
\node[inner sep=0.5,draw,shape=circle] at (0.5,-0.5) {\small $1$};
\node[inner sep=0.5,draw,shape=circle] at (-0.5,-0.5) {\small $4$};
\node[inner sep=0.5,draw,shape=circle] at (-0.5,0.5) {\small $3$};
\node[inner sep=0.5,draw,shape=circle] at (1.45,0) {\small $5$};
\node[inner sep=0.5,draw,shape=circle] at (0,-1.45) {\small $6$};
\node[inner sep=0.5,draw,shape=circle] at (1.3,-1.3) {\small $7$};
\node[inner sep=0.5,draw,shape=circle] at (-1.3,-1.3) {\small $8$};
\node[inner sep=0.5,draw,shape=circle] at (0,-2.15) {\small $9$};
\node[inner sep=0,draw,shape=circle] at (1.3,-2.15) {\footnotesize $10$};
\node[inner sep=0,draw,shape=circle] at (2.15,0) {\footnotesize $11$};
\node[inner sep=0,draw,shape=circle] at (2.85,0) {\footnotesize $12$};
\node[inner sep=0,draw,shape=circle] at (1.25,1.25) {\footnotesize $14$};
\node[inner sep=0,draw,shape=circle] at (2.5,1.25) {\footnotesize $13$};

\end{scope}

\end{tikzpicture}
\caption{Tilings for $\{\alpha\beta\gamma,\delta^3,\epsilon^4\}$ and $\{\alpha\beta\gamma,\delta^3,\delta\epsilon^3\}$.}
\label{prooffig}
\end{figure}

Now we consider the AVC $\{
24\alpha\beta\gamma\delta\epsilon \colon 
24\alpha\beta\gamma,
8\delta^3,
6\epsilon^4 \}$ in \eqref{24a}. We start with a vertex $\epsilon^4$ surrounded by tiles $T_1,T_2,T_3,T_4$ in the second of Figure \ref{prooffig}. We may assume $T_1$ is arranged as indicated. By $\alpha_1\cdots=\alpha\beta\gamma$ and $\beta,\epsilon$ not adjacent in $T_2$, we get $\beta_5,\gamma_2$. Then $\gamma_2,\epsilon_2$ determine (the angle arrangement of) $T_2$. Then we may start with $\alpha_2\cdots$ in place of $\alpha_1\cdots$ and repeat the argument. This determines $T_3$. One more repeat of the argument with $\alpha_3\cdots$ further determines $T_4$. On the other hand, by $\beta_2\cdots=\alpha\beta\gamma$ and $\alpha,\beta$ not adjacent in $T_3$, we get $\alpha_6,\gamma_5$. Then $\beta_5,\gamma_5$  determine $T_5$. Then $\epsilon_5\cdots=\epsilon^4$ and $\alpha_6$ determine $T_6$. We also note that $\delta_1\cdots=\delta_2\cdots=\delta^3$. 

The argument that started with the initial $\epsilon^4$ can be applied to $\epsilon_5\epsilon_6\cdots=\epsilon^4$, and determines more tiles. More repetitions of the argument gives the pentagonal subdivision of the octahedron in the second of Figure \ref{12angletiling}. 

For the AVC $\{
60\alpha\beta\gamma\delta\epsilon \colon 
60\alpha\beta\gamma,
20\delta^3,
12\epsilon^5 \}$ in \eqref{60a}, we may carry out the same argument with $\epsilon^5$ in place of $\epsilon^4$. We get the pentagonal subdivision of the icosahedron.

It remains to consider the AVC $\{
36\alpha\beta\gamma\delta\epsilon \colon 
36\alpha\beta\gamma,\,
8\delta^3,\,
12\delta\epsilon^3 \}$ in \eqref{36a}. We start with $T_1,T_2,T_3,T_4$ around $\delta\epsilon^3$ in the third of Figure \ref{prooffig}. We may assume the angles of $T_1$ are arranged as indicated. Then we determine $T_2,T_5$ by the same reason as the second picture. By $\alpha_2\cdots=\alpha\beta\gamma$ and $\gamma,\delta$ not adjacent, we get $\beta_3$. Then $\beta_3,\delta_3$ determine $T_3$. By $\alpha_3\cdots=\alpha\beta\gamma$ and $\beta,\epsilon$ not adjacent, we get $\gamma_4$. Then $\gamma_4,\epsilon_4$ determine $T_4$. Then $\alpha_4\gamma_1\cdots=\alpha\beta\gamma$ gives $\beta_6$. Then $\beta_1\cdots=\alpha\beta\gamma$ and $\alpha,\beta$ not adjacent determine $\gamma_6,\alpha_7$. Then $\beta_6,\gamma_6$ determine $T_6$, and $\alpha_7$ and $\delta_1\delta_5\cdots=\delta^3$ determine $T_7$. Moreover, $\delta_4\delta_6\cdots=\delta_3$, and $\beta_4\cdots=\alpha\beta\gamma$, and $\gamma,\delta$ not adjacent determine $T_8$. Then $\alpha_6\beta_8\cdots=\alpha\beta\gamma$, and $\epsilon_6\epsilon_7\cdots=\delta\epsilon^3$, and $\gamma,\delta$ not adjacent determine $T_9$ and give $\delta_{10}$. 

By $\alpha_5\beta_7\cdots=\alpha\beta\gamma$ and no $\gamma_7\epsilon\cdots$, we determine $T_{11}$. Then $\delta_{10}$ and $\beta_{11}\gamma_7\cdots=\alpha\beta\gamma$ determine $T_{10}$. Then $\delta_{11}\epsilon_{10}\cdots=\delta\epsilon^3$ gives $\epsilon_{12}$. Then $\alpha_{11}\cdots=\alpha\beta\gamma$ and $\beta,\epsilon$ not adjacent give $\beta_{13},\gamma_{12}$. Then $\alpha_{11}\gamma_{12}\cdots=\alpha\beta\gamma$, and $\beta_2\gamma_5\cdots=\alpha\beta\gamma$, and $\epsilon_5\epsilon_{11}\cdots=\delta\epsilon_3$, and $\beta,\epsilon$ not adjacent determine $T_{13},T_{14}$. Then $\alpha_2\beta_3\cdots=\alpha\beta\gamma$ and $\delta_2\delta_{13}\cdots=\delta^3$ imply $\gamma,\delta$ adjacent, a contradiction. 
\end{proof}

We further examine how the tilings in Theorem \ref{anglerealize} can be geometrically realized, in the sense all tiles are congruent, i.e., transformed to each other by rotations and flips of the sphere. For the pentagonal subdivision of the octahedron in the second of Figure \ref{12angletiling}, this means $\alpha\delta$-edge and $\beta\delta$-edge are identified, and $\alpha\epsilon$-edge and $\gamma\epsilon$-edge are identified. If these edges are straight, then we get the tiling on the left of Figure \ref{geomtiling}. The pentagon is given by the upper right picture, with $\alpha\delta$- and $\beta\delta$-edges indicated by the normal line, and $\alpha\epsilon$- and $\gamma\epsilon$-edges indicated by the thick line, and the $\beta\gamma$-edge indicated by the dashed line. This is the pentagonal subdivision of the octahedron in \cite{wy1}, and the 2-dimensional moduli is described in detail in \cite{wy3}.

\begin{figure}[htp]
\centering
\begin{tikzpicture}[>=latex]


\foreach \a in {0,...,3}
{
\begin{scope}[rotate=90*\a]

\draw
	(0,0.8) -- (0,0) -- (0.8,0) -- (1.2,0.6) -- (0.6,1.2) 
	(1.2,0.6) -- (1.8,0.6)
	(1.2,-0.6) -- (2.6,-0.6)
	(1.8,-2.6) -- (1.8,0.6)
	(2.6,-0.6) -- (2.6,2.6) -- (3,3)
	(0.4,0) -- ++(-45:0.2)
	(1,0.3) -- ++(-80:0.2)
	(0.9,0.9) -- ++(180:0.2)
	(0.933,-0.2) -- ++(210:0.2)
	(1.067,-0.4) -- ++(30:0.2)
	(1.5,0.6) -- ++(45:0.2)
	(1.8,0) -- ++(45:0.2)
	(1.5,-0.6) -- ++(135:0.2)
	(1.2,1.8) -- ++(-45:0.2)
	(1.8,1) -- ++(0:0.2)
	(1.8,1.4) -- ++(180:0.2)
	(2.2,-0.6) -- ++(-45:0.2)
	(2.6,0.6) -- ++(-45:0.2)
	(2.2,1.8) -- ++(225:0.2)
	(2.6,2.2) -- ++(135:0.2)
	(2.6,-1.3) -- ++(180:0.2)
	(2.6,-1.9) -- ++(0:0.2)
	(2.8,2.8) -- ++(180:0.2);

\draw[line width=1.2]
	(0,0) -- (0.8,0)
	(1.8,1.8) -- (-0.6,1.8) 
	(0.6,1.2) -- (0.6,2.6)
	(2.6,2.6) -- (3,3);

\draw[dashed]
	(0.6,1.2) -- (0,0.8)
	(1.8,0.6) -- (1.8,1.8)
	(0.6,2.6) -- (2.6,2.6);
	
\end{scope}
}

\begin{scope}[xshift=5.5cm]

\begin{scope}[yshift=1.5cm]

\draw
	(90:1) -- (162:1) -- (234:1)
	(54:0.81) -- ++(94:0.3)
	(-18:0.81) -- ++(202:0.3)
	(126:0.81) -- ++(86:0.3)
	(198:0.81) -- ++(-22:0.3)
	(0.2,-0.81) -- ++(0,0.3)
	(-0.2,-0.81) -- ++(0,-0.3);
\draw[line width=1.2]
	(90:1) -- (18:1) -- (-54:1);
\draw[dashed]
	(234:1) -- (-54:1);

\node at (90:0.75) {\small $\alpha$};
\node at (234:0.75) {\small $\beta$};
\node at (-54:0.75) {\small $\gamma$};
\node at (162:0.75) {\small $\delta$};
\node at (18:0.75) {\small $\epsilon$};

\end{scope}

\begin{scope}[yshift=-1.5cm]

\draw
	(90:1) to[out=174, in=46] 
	(162:1) to[out=-62, in=68] (234:1);
\draw[line width=1.2]
	(90:1) to[out=-56, in=84] 
	(18:1) to[out=-158, in=52] (-54:1);
\draw[dashed]
	(234:1) to[out=-40, in=140] (-54:1);

\end{scope}

\end{scope}

\end{tikzpicture}
\caption{Geometrically congruent tiling for $f=24$.}
\label{geomtiling}
\end{figure}

However, for angle congruent tilings, there is no reason we should restrict ourselves to straight line edges. In fact, Heesch and Kienzle \cite{hk} studied curvilinear polygons that can tile the plane. On the lower right of Figure \ref{geomtiling}, we describe the curvilinear pentagon that can be the tile in the pentagonal subdivision. The key point is that the pentagon has ``side neighborhoods'' along each of its five edges. Two such side neighborhoods are on the two sides of the identified (normal) $\alpha\delta$- and $\beta\delta$-edge. From the viewpoint of the pentagon, the $\beta\delta$-side neighborhood is the outside of the $\alpha\delta$-side neighborhood. Moreover, the $\delta$-ends of the two sides are matched. The same happens to the (thick) $\alpha\epsilon$- and $\beta\epsilon$-side neighborhoods. 

As for the (dashed) $\beta\gamma$-edge, we see two copies of the $\beta\gamma$-side neighborhood are matched. This self matching means the $\beta\gamma$-edge must have the $180^{\circ}$-rotation symmetry. 

We indicate the matchings of the edge pairs by adding small edge markers. Note that the marker on the $\beta\gamma$-edge by two small edges represent the $180^{\circ}$-rotation symmetry, and the markers on the other edges indicate no symmetry. Then the tiling on the left must also carry the markers in compatible way. Of course, the edges may have more symmetry than indicated by the markers, with straight being the most symmetric. However, this does not change the pentagonal subdivision tiling being the only tiling for the AVC \eqref{24a}.

The discussion about the geometric congruent tiling, including curvilinear tiles, also applies to the pentagonal subdivision of the icosahedron derived from the AVC \eqref{60a}. 

The discussion also applies to the pentagonal subdivision of the tetrahedron on the first of Figure \ref{12angletiling}, provided the angles are located as in the picture. We remark that the two $\alpha\delta$-edges are not required to be isometric. We only need $\beta\delta$-edge to match (one) connected $\alpha\delta$-edge, and $\gamma\delta$-edge to match (another) connected $\alpha\delta$-edge. 

If we exchange $\beta,\gamma$ along some dashed edges, then this may cause the matching of $\beta\delta$-edge to the non-connected $\alpha\delta$-edge, and the matching of $\gamma\delta$-edge to the non-connected $\alpha\delta$-edge. Such matching is described by the first of Figure \ref{geomtiling2}. However, by \cite[Proposition 7]{gsy} and \cite[Lemma 9]{wy1}, there is no tiling for such tile, unless the normal and thick edges are also matched. Therefore both the matching in the first of Figure \ref{geomtiling2} and the matching in Figure \ref{geomtiling} should happen. This means the pentagon is the second of Figure \ref{geomtiling2}. The pentagon appears to be symmetric. But we do not need to have $\beta=\gamma$ because the $\beta\gamma$-edge is only required to be $180^{\circ}$-rotation symmetric. 

\begin{figure}[htp]
\centering
\begin{tikzpicture}[>=latex]

\draw
	(90:1) -- (162:1)
	(18:1) -- (-54:1);
\draw[line width=1.2]
	(90:1) -- (18:1)
	(162:1) -- (234:1);
\draw[dashed]
	(234:1) -- (-54:1);

\node at (90:0.75) {\small $\alpha$};
\node at (234:0.75) {\small $\beta$};
\node at (-54:0.75) {\small $\gamma$};
\node at (162:0.75) {\small $\delta$};
\node at (18:0.75) {\small $\delta$};

\begin{scope}[xshift=3cm]

\draw
	(90:1) -- (162:1) -- (234:1)
	(90:1) -- (18:1) -- (-54:1)
	(54:0.81) -- ++(94:0.3)
	(-18:0.81) -- ++(202:0.3)
	(126:0.81) -- ++(86:0.3)
	(198:0.81) -- ++(-22:0.3)
	(0.2,-0.81) -- ++(0,0.3)
	(-0.2,-0.81) -- ++(0,-0.3);
\draw[dashed]
	(234:1) -- (-54:1);

\node at (90:0.75) {\small $\alpha$};
\node at (234:0.75) {\small $\beta$};
\node at (-54:0.75) {\small $\gamma$};
\node at (162:0.75) {\small $\delta$};
\node at (18:0.75) {\small $\delta$};

\end{scope}

\begin{scope}[xshift=6cm]

\foreach \a in {1,-1}
\foreach \b in {1,-1}
{
\begin{scope}[yshift=0.81*\b cm, yscale=\b, xscale=\a]

\draw
	(90:1) to[out=-56, in=84] 
	(18:1) to[out=-158, in=52] (-54:1);
\fill
	(18:1) circle (0.05);
		
\end{scope}	
}
	
\draw[yshift=0.81cm, dashed]
	(234:1) to[out=-40, in=140] (-54:1);

\end{scope}

\begin{scope}[xshift=9cm]

\draw[dotted]
	(-0.8,-0.8) rectangle (0.8,0.8);
	
\foreach \a in {1,-1}
\foreach \b in {1,-1}
{
\begin{scope}[xscale=\a, yscale=\b]

\draw
	(0.5,0) -- (0.8,0.8) -- (0,1.1);

\fill 
	(0.8,0.8) circle (0.05);
	
\end{scope}
}

\draw[dashed]
	(0.5,0) -- (-0.5,0);

\draw[gray]
	(0.8,0.8) -- (1.6,0.5) -- (2.4,0.8) -- (2.7,0) -- (2.4,-0.8) -- (1.6,-0.5) -- (0.8,-0.8);

\draw[gray,dashed]
	(1.6,0.5) -- (1.6,-0.5);

\end{scope}

\end{tikzpicture}
\caption{Exchange of $\beta,\gamma$ in the dodecahedron tiling.}
\label{geomtiling2}
\end{figure}

The whole tiling is the tiling of the pairs obtained by glueing two tiles along the $\beta\gamma$-edge. The general picture for the pair is the third of Figure \ref{geomtiling2}. The whole tiling is actually obtained by gluing six pairs together. In fact, each pair is the modification of a face of the regular cube, with four $\bullet$ vertices matching the four corners of the square face. The fourth of Figure \ref{geomtiling2} gives the scheme for the construction, which also appears in \cite[Figure 55]{cly4}.

\end{document}